\documentclass[12pt,reqno]{amsart}

\usepackage{amsmath,amsfonts,amssymb,amsthm,graphicx,fullpage}

\numberwithin{equation}{section}

\newtheorem{theorem}{Theorem}[section]

\newtheorem{lemma}[theorem]{Lemma}
\newtheorem{corollary}[theorem]{Corollary}

\theoremstyle{definition}

\theoremstyle{remark}
\newtheorem{example}[theorem]{Example}
\newtheorem{remark}[theorem]{Remark}

\newcommand{\E}{\mathbf E}
\newcommand{\R}{\mathbb R}
\newcommand{\Rd}[1][d]{\R^{#1}}

\renewcommand{\SS}{\mathbb S}
\newcommand{\Sphere}[1][d-1]{\SS^{#1}}

\newcommand{\sA}{\mathcal A}

\newcommand{\sG}{\mathcal G}
\newcommand{\sX}{\mathcal X}

\newcommand{\sH}{\mathcal H}
\newcommand{\sF}{\mathcal F}

\newcommand{\sR}{\mathcal R}

\newcommand{\sI}{\mathcal I}

\newcommand{\one}{\mathbf{1}}

\newcommand{\Lp}[1][p]{\mathsf{L}^{#1}}

\newcommand{\Hx}[1][x]{H_{u_{#1}}}

\DeclareMathOperator{\conv}{conv}

\DeclareMathOperator{\cl}{cl}

\DeclareMathOperator{\proj}{proj}

\renewcommand{\epsilon}{\varepsilon}
\renewcommand{\phi}{\varphi}
\newcommand{\eps}{\varepsilon}

\newcommand{\emp}[1][n]{\widehat{\sR}_{#1}}
\newcommand{\emm}[1][n]{\widehat{\mu}_{#1}}
\newcommand{\empm}[1][n]{\hat{\theta}_{#1}}
\newcommand{\grad}{\operatorname{grad}}

\usepackage{fancybox}
\newlength{\querylen}
\begin{document} 

\title{Strong limit
  theorems for empirical halfspace depth trimmed regions}

\author{Andrii Ilienko, Ilya Molchanov, Riccardo Turin}

% \begin{frontmatter}

%   \runtitle{Limit theorems for depth trimmed regions}
  
%   \begin{aug}
    
%     \author[A,B]{\inits{A.}\fnms{Andrii}~\snm{Ilienko}\ead[label=e1]
%       {andrii.ilienko@unibe.ch}}
%     \author[A]{\inits{I.}\fnms{Ilya}~\snm{Molchanov}\ead[label=e2]
%       {ilya.molchanov@unibe.ch}}
%     \author[A,C]{\inits{R.}\fnms{Riccardo}~\snm{Turin}\ead[label=e3]
% %      {riccardo.turin@stat.unibe.ch}}
%       {Riccardo\_Turin@swissre.com}}
    
%     \address[A]{Institute of Mathematical Statistics and Actuarial
%       Science, University of Bern, Alpeneggstrasse 22, CH-3012 Bern,
%       Switzerland\printead[presep={,\ }]{e1,e2,e3}}
%     \address[B]{National Technical University of Ukraine \lq\lq Igor
%       Sikorsky Kyiv Polytechnic Institute\rq\rq, Beresteiskyi
%       prosp.~37, 03056, Kyiv, Ukraine}
%     \address[C]{Swiss Re Management
%       Ltd, Mythenquai 50/60, 8022 Zurich, Switzerland}
%   \end{aug}

  \begin{abstract}
    We study empirical variants of the halfspace (Tukey) depth of a
    probability measure $\mu$, which are obtained by replacing $\mu$
    with the corresponding weighted empirical measure. We prove
    analogues of the Marcinkiewicz--Zygmund strong law of large
    numbers and of the law of the iterated logarithm in terms of set
    inclusions and for the Hausdorff distance between the theoretical
    and empirical variants of depth trimmed regions. In the special
    case of $\mu$ being the uniform distribution on a convex body $K$,
    the depth trimmed regions are convex floating bodies of $K$, and
    we obtain strong limit theorems for their empirical estimators.
  \end{abstract}
	
%   \begin{keyword}
%     \kwd{convex floating body}
%     \kwd{empirical measure}
%     \kwd{halfspace depth}
%     \kwd{law of the iterated logarithm}
%     \kwd{strong law of large numbers}
%   \end{keyword}
	
% \end{frontmatter}

  \maketitle

  \renewcommand{\thefootnote}{}
  
  \phantom{1}\footnote{
  AI, IM: Institute of Mathematical Statistics and Actuarial
      Science, University of Bern, Alpeneggstrasse 22, CH-3012 Bern,
      Switzerland
      
  AI: National Technical University of Ukraine \lq\lq Igor
      Sikorsky Kyiv Polytechnic Institute\rq\rq, Beresteiskyi
      prosp.~37, 03056, Kyiv, Ukraine
      
  RT: Swiss Re Management
      Ltd, Mythenquai 50/60, 8022 Zurich, Switzerland}

\section{Introduction}
\label{sec:introduction}

The \emph{halfspace depth} of a point $x\in\Rd$ with respect to a
Borel probability measure $\mu$ on $\R^d$, also called the \emph{depth
  function}, is defined as
\begin{equation}
  \label{eq:depth}
  D(x)=\inf\{\mu(H)\colon H\in\sH, H\ni x\}, \quad x\in\R^d,  
\end{equation}
where $\sH$ is the family of all closed halfspaces in $\R^d$. A
generic halfspace is denoted by $H$ and its boundary by $\partial H$,
which is $(d-1)$-dimensional affine subspace. In other words, the
depth of $x$ is the smallest measure of a halfspace containing $x$.
Clearly, the infimum in \eqref{eq:depth} can be equivalently taken
over all $H$ with $x\in\partial H$.  This concept was first introduced
by Tukey \cite{tuk75} and is also called the \emph{Tukey depth}, see
\cite{nag:sch:wer18} for a survey of recent results and geometric
connections.  For $\alpha\in(0,1]$, define the \emph{depth trimmed
  region}
\begin{displaymath}
  \sR(\alpha)=\{x\in\R^d\colon D(x)\geq \alpha\},
\end{displaymath}
that is, $\sR(\alpha)$ is the upper level set of the depth function.
The regions $\sR(\alpha)$ are nested (i.e., decreasing with respect to
the set inclusion) compact convex subsets of $\R^d$, see Corollary
after Proposition~1 in \cite{rous:rut99} and Proposition~5
therein. Note that $\sR(\alpha)$ converges to the convex hull of the
support of $\mu$ as $\alpha\downarrow0$.

Obviously, $\sR(\alpha)$ is either empty or a singleton when
$\alpha>1/2$. Moreover, for any $\alpha$,
\begin{equation}
  \label{eq:pop}
  \sR(\alpha)=\bigcap_{H\in\sH\colon \mu(H)>1-\alpha} H,
\end{equation}
see Proposition~6 and Corollary thereafter in \cite{rous:rut99}.

An empirical version $\emp(\alpha)$ of $\sR(\alpha)$ is constructed by
replacing $\mu$ with the \emph{weighted empirical measure}
\begin{equation}
  \label{eq:emm}
  \emm=n^{-1}\sum_{i=1}^{n}\xi_i\,\delta_{X_i}\,,
\end{equation}
where $\{X_i,i\ge1\}$ is a sequence of $\mu$-distributed independent
random vectors, $\{\xi_i,i\ge1\}$ is a sequence of i.i.d.~random
variables with $\E\xi_1=1$ independent of $\{X_i,i\ge1\}$, and
$\delta_x$ denotes the Dirac measure at $x\in\Rd$. In particular,
$\emm(\Rd)=n^{-1}\sum_{i=1}^{n}\xi_i=\bar\xi_n$, the sample mean of
$\xi_i$.  If the weights $\xi_i$ are chosen to be identically one,
$\emm$ is the \emph{empirical measure}, see, e.g., \cite{shor:wel86}
or \cite{VaartWellner}. In general, $\emm$ might be also signed. The
above construction of the depth function and the depth trimmed regions
applied to the empirical measure $\emm$ yields the empirical depth and
the \emph{empirical depth trimmed regions} $\emp(\alpha)$. It can be
easily shown that \eqref{eq:pop} remains true for finite signed
measures and is applicable to $\emp(\alpha)$ and $\emm$ with
$1-\alpha$ replaced by $\bar\xi_n-\alpha$.

Working with the weighted variant of empirical measures makes it
possible to cover the case of possibly negative weights, which may
destroy the monotonicity property of $\emp(\alpha)$ in
$\alpha$. Particularly complicated effects arise when the expected
weight approaches zero, which we leave out of the scope of the current
work. Furthermore, weighted empirical measures arise from the
resampling scheme and it is likely that they would lead to bootstrap
estimators of the depth-trimmed regions.

Our results on convergence of empirical depth trimmed regions are
split into two groups. The nearness of $\sR(\alpha)$ and
$\emp(\alpha)$ can be assessed in the inclusion sense by sandwiching
the empirical depth trimmed region between two versions
$\sR(\alpha'_n)$ and $\sR(\alpha''_n)$ with parameters $\alpha'_n$ and
$\alpha''_n$ which converge to $\alpha$ from below and above as
$n\to\infty$. In this case, we study the asymptotics of differences
between $\alpha'_n$, $\alpha''_n$ and $\alpha$. Another way is to
consider the asymptotical behaviour of the Hausdorff distance
$\rho_H(\sR(\alpha),\emp(\alpha))$. Strong limit theorems in the
metric sense typically require additional assumptions on $\mu$, which
are of smoothness nature.

In the special case when the measure $\mu$ is uniform on a convex body
$K\subset\Rd$, the set $\sR(\alpha)$ is the convex floating body of
$K$ at level $\alpha V_d(K)$, where $V_d(\cdot)$ denotes the Lebesgue
volume, see \cite{nag:sch:wer18} and \cite{schut:wer90}.  Then
$\emp(\alpha)$ becomes an empirical estimator of the convex floating
body. The Hausdorff distance between $\emp(\alpha)$ and $\sR(\alpha)$
in this setting has been studied in \cite{bru19}, where concentration
type results are obtained, proving an upper exponential bound on the
tail of the Hausdorff distance. The Hausdorff and Banach--Mazur
distances between the convex hull of $n$ points and the floating body
at level $1/n$ was studied in \cite{fres13}. Complementing these
works, we show that empirical floating bodies approximate the
theoretical one at the rate $(n^{-1}\log\log n)^{1/2}$, so that the
rate $n^{-1/2}$ for the convergence in probability from
\cite[Corollary~2]{bru19} does not hold almost surely. It is
interesting that this rate does not depend on the dimension, while the
rate of approximation of a $C^2$-smooth convex body $K$ by the convex
hull of a uniform random sample in $K$ is $(n^{-1}\log n)^{2/(d+1)}$
and so depends on dimension, see \cite[Theorem~6]{bar:89}.

From the central limit theorem for depth proved in \cite{masse04},
applying the inversion tools from \cite{mol98}, one arrives at a
central limit theorem for the Hausdorff distance between
$\emp(\alpha)$ and $\sR(\alpha)$. We leave these results and an
eventual functional limit theorem out of the scope of the current
work.

It should be noted that related problems can be considered also for
unbounded measures. For example, for the Lebesgue measure on $\R_+^d$
and a variant of empirical depth trimmed regions derived from multiple
sums on the integer lattice, the strong law of large numbers and the law of the
iterated logarithm were obtained in \cite{il-mol:18}.

The paper is organised as follows.
Section~\ref{sec:strong-limit-theor} reminds basic properties of the
depth and standard assumptions on the underlying probability
measure. It also contains some necessary representations of the
theoretical and empirical depth trimmed regions.
In Section \ref{sec:SLLN}, we prove almost sure convergence of the
empirical halfspace depth trimmed regions towards the theoretical ones
and establish the rate of this convergence in terms of the
Marcinkiewicz--Zygmund law.  Section~\ref{sec:LIL} contains the law of
the iterated logarithm. Both these sections present results in the
inclusion setting. Among other methods, the proofs extensively use
techniques from the theory of empirical measures.

Section~\ref{sec:diff-depth-pert} presents a deterministic result
concerning asymptotics of the Hausdorff distance between two
depth trimmed regions at close levels. Since depth trimmed regions
appear as solutions of inequalities, we can take advantage of the
approach used in \cite{mol98} to handle this task. Assuming
differentiability of the depth function, we determine its gradient in
terms of the Radon transform of the underlying measure. 
In Section~\ref{sec:strong-limit-theor-Hausdorff}, we prove the strong
law of large numbers and the law of the iterated logarithm for the
Hausdorff distance between $\emp(\alpha)$ and $\sR(\alpha)$.  In
Section~\ref{sec:ex}, we present several examples. 

Appendix contains some geometric results which make it possible to
show that the depth function generated by the uniform distribution on
a symmetric smooth strictly convex body is differentiable.

\section{Basic properties of depth trimmed regions}
\label{sec:strong-limit-theor}

We start by listing some conditions on the measure $\mu$ needed in
what follows. The first two are quite standard.
\begin{enumerate}
\item[(C1)] $\mu(E)=0$ for each hyperplane $E$.
\item[(C2)] If $H_1,H_2\in\sH$ are disjoint with
  $\mu(H_1),\mu(H_2)>0$, then $\mu(H_1)+\mu(H_2)<1$.
\end{enumerate}

Condition (C1) holds if $\mu$ is absolutely continuous with
respect to the $d$-dimensional Lebesgue measure.  Under (C1), the
infimum in \eqref{eq:depth} is attained at some (possibly not unique)
\emph{minimal halfspace} $H$, and the depth function $D$ is continuous
in $x$, see Proposition~4.5 in \cite{masse04} or Proposition~1 in \cite{miz02}.

Condition (C2) is usually called the contiguity of support. This means
that the support of $\mu$ cannot be divided into two disjoint parts
separated by a strip. Clearly, (C2) holds for $\mu$ with connected
support. Under (C1) and (C2), the map $\alpha\mapsto\sR(\alpha)$ is
continuous in the Hausdorff metric $\rho_H$ on
$\alpha\in(0,\alpha_{\max})$, where
\begin{math}
  \alpha_{\max}=\sup\{D(x)\colon x\in\Rd\},
\end{math}
see Theorem~1 in \cite{nag:sch:wer18}. Under some additional
conditions, the differentiability of the function
$\rho_H(\sR(\alpha),\sR(\alpha+t))$ in $t$ is proved in
Section~\ref{sec:diff-depth-pert}.

In the following lemma, we give two further useful consequences of
(C1) and (C2). For a set $K$ in $\R^d$ and $E=\partial H$ for a
halfspace $H$, we say that $E$ is a \emph{supporting hyperplane} of
$K$ at $x\in\partial K$ if $x\in E$ and $K$ is contained in one of the
two closed halfspaces bounded by $E$.

\begin{lemma}
  \label{lem:level}
  Under (C1), (C2), and for each $\alpha\in(0,\alpha_{\max})$,
  \begin{enumerate}
  \item[(i)] $\partial\sR(\alpha)=\{x\colon D(x)=\alpha\}$, which is the
    $\alpha$-level surface of the depth function $D$;
  \item[(ii)] the boundary of each minimal halfspace $H$ at
    $x\in\partial\sR(\alpha)$ is a supporting hyperplane
    of~$\sR(\alpha)$.
  \end{enumerate}
\end{lemma}
\begin{proof}
  If $x\in\partial\sR(\alpha)$ then $x_i^-\to x$ and $x_i^+\to x$ for
  some sequences $\{x_i^-,i\ge1\}$ and $\{x_i^+,i\ge1\}$ with
  $D(x_i^-)<\alpha$ and $D(x_i^+)\ge\alpha$ for all $i$. Hence,
  $D(x)=\alpha$ by continuity of the depth function implied by (C1).
  
  Conversely, if $D(x)=\alpha$ then $x\in\partial\sR(\alpha)$ since, otherwise,
  we would have $x\in\sR(\beta)$ for some $\beta>\alpha$ by
  continuity of $\sR(\cdot)$ in the Hausdorff metric implied by (C1) and
  (C2). This contradicts to $D(x)=\alpha$.
  
  Finally, if $\partial H$ does not support $\sR(\alpha)$ then
  $G\subset H$ for some $G\in\sH$ such that $\partial G$ supports
  $\sR(\alpha)$ at some point $y\in\partial\sR(\alpha)$. Hence
  $\mu(G)\ge D(y)=\alpha$. On the contrary, $G\subset H$ and (C2)
  yield $\mu(G)<\mu(H)=D(x)=\alpha$.
\end{proof}

In \eqref{eq:pop} and in a similar representation of $\emp(\alpha)$,
we deal with uncountable intersections of closed halfspaces.  In case
of $\emp(\alpha)$, the condition $\emm(H)>\bar\xi_n-\alpha$ is a
random event due to measurability of $\emm$. Hence, to guarantee the
measurability of the intersection we should show that $\emp(\alpha)$
are countably determined. This means that we can replace $\sH$ in the
above representation by a countable subclass $\sH_0\subset\sH$, and
still almost surely obtain the same depth trimmed regions.

The family of halfspaces whose boundaries contain the origin
can be parametrised as
\begin{equation*}
  H_u=\{z\in\R^d\colon \langle z,u\rangle \le 0\},\quad u\in\Sphere,
\end{equation*}
where $\Sphere$ stands for the unit sphere in $\Rd$.  Each halfspace
from $\sH$ can be given as $H_u+x$ for some $x\in\R^d$ and
$u\in\Sphere$ or as $H_u(t)=\{z\in\R^d\colon \langle z,u\rangle \le
t\}$ for $u\in\Sphere$ and $t\in\R$.

\begin{lemma}
  \label{th:count}
  Let $\sH_0$ be the family of halfspaces $H_u(t)$ for all rational
  numbers $t$ and all $u$ from a countable dense set on the unit
  sphere.  If $\mu$ is a finite signed measure then
  \begin{equation}
    \label{eq:countable}
    \sR(\alpha)
    =\cap \bigl\{H\in\sH_0\colon \mu(H)>\mu(\Rd)-\alpha\bigr\}.
    % =\bigcap_{\substack{H\in\sH_0\colon\\\mu(H)>\mu(\Rd)-\alpha}} H.
  \end{equation}
  In particular, $\emp(\alpha)$ is a random closed set.
\end{lemma}
\begin{proof}
  It suffices to prove that $\sR(\alpha)$ contains the
  right-hand side of \eqref{eq:countable},
  since the opposite inclusion is obvious.  If $x$ does not belong to
  $\sR(\alpha)$, then $x\notin H$ for some $H\in\sH$ such that
  $\mu(H)>\mu(\Rd)-\alpha+\eps$ with $\varepsilon>0$. Let
  $\mu=\mu^+-\mu^-$ be the Jordan decomposition of $\mu$. Consider a
  ball $B\subset\Rd$ such that $\mu^+(B^c)\le\eps/4$ and
  $\mu^-(B^c)\le\eps/4$. It follows from the definition of $\sH_0$
  that there exist $H_n\in\sH_0$, $n\ge1$, such that $x\notin H_n$,
  $H_{n+1}\cap B\subset H_n\cap B$, and
  $\bigcap_{n=1}^\infty(H_n\cap B)=H\cap B$. Then
  \begin{displaymath}
    \lim_{n\to\infty}\mu^+(H_n\cap B)=\mu^+(H\cap B),\qquad
    \lim_{n\to\infty}\mu^-(H_n\cap B)=\mu^-(H\cap B).
  \end{displaymath}
  Hence, for large $n$, 
  \begin{gather*}
    \mu^+(H_n)\ge\mu^+(H_n\cap B)\ge\mu^+(H\cap B)-\eps/4\ge
    \mu^+(H)-\eps/2,\\
    \mu^-(H_n)\le\mu^-(H_n\cap B)+\eps/4\le\mu^-(H\cap B)+\eps/2\le
    \mu^-(H)+\eps/2.
  \end{gather*}
  Thus, for such $n$, $\mu(H_n)\ge\mu(H)-\eps>\mu(\Rd)-\alpha$,
  meaning that $x$ does not belong to the right-hand side of
  \eqref{eq:countable}. This proves the first claim.
	
  Let $H(n,\alpha)$ be $H$ if $\emm(H)>\bar\xi_n-\alpha$ and $\Rd$
  otherwise. By \eqref{eq:countable}, $\emp(\alpha)$ is the
  (countable) intersection of $H(n,\alpha)$ obtained for all
  $H\in\sH_0$.  Since $H(n,\alpha)$ are simple random closed sets in
  the sense of \cite[Definition~1.3.15]{mol17}, the second claim
  follows from Theorem~1.3.25(vii) therein.
\end{proof}

The following result clarifies the structure of empirical depth
trimmed regions. A related result is available in Corollary~13 of
\cite{MR4426418}. 

\begin{lemma}
  \label{lemma:boundary}
  The boundary of the empirical depth trimmed region $\emp(\alpha)$ is
  contained in the union of all affine subspaces generated by
  $i$-tuples of points from $\{X_1,\dots,X_n\}$ with $i=1,\ldots,d$.
\end{lemma}
\begin{proof}
  Let $\sI$ be the family of all subsets $I\subset \Xi_n=\{X_1,\dots,X_n\}$
  such that $I$ equals the intersection of its convex hull $\conv I$
  with $\Xi_n$ and $n^{-1}\sum_{X_i\in I} \xi_i>\bar\xi_n-\alpha$.

  If $I\in\sI$ and $I\subset H$ for a halfspace $H$ and
  $H\cap(\Xi_n\setminus I)=\emptyset$, then $\emm(H)>\bar\xi_n-\alpha$.  If
  $\emm(H)>\bar\xi_n-\alpha$ for a halfspace $H$, then $H\supset I$ for
  $I=H\cap\Xi_n\in \sI$ and $H$ does not intersect $\Xi_n\setminus
  I$. Thus,
  \begin{align*}
    \emp(\alpha)&=\cap\{H\colon\emm(H)>\bar\xi_n-\alpha\}
    =\cap\{H\colon H\supset I, H\cap (\Xi_n\setminus I)=\emptyset
    \;\text{for}\; I\in\sI\}\\
    &=\bigcap_{I\in\sI} \Big(\cap\{H\colon H\supset I, H\cap (\Xi_n\setminus
    I)=\emptyset\}\Big).
  \end{align*}
  Hence, the boundary of $\emp(\alpha)$ is a subset of the union of
  boundaries of the sets
  \[\cap\{H\colon H\supset I, H\cap (\Xi_n\setminus I)=\emptyset\}\]
  for $I\in\sI$, not to say for all subsets $I$ of $\Xi_n$. The
  statement now follows from the following lemma. 
\end{proof}

\begin{lemma}
  If $I_1$ and $I_2$ are finite sets in $\R^d$ then the intersection
  $A$ of all halfspaces which contain $I_1$ and do not intersect $I_2$
  is either $\Rd$ or has the boundary which is a subset of the union
  of all affine subspaces generated by $i$-tuples of points from
  $I_1\cup I_2$ with $i=1,\ldots,d$.
\end{lemma}
\begin{proof}
  If $\conv I_1\hspace{.5pt}\cap\hspace{.5pt}\conv I_2\ne\emptyset$,
  there are no halfspaces which contain $I_1$ and do not intersect
  $I_2$. In this case $A=\Rd$.
	
  Assume that the convex hulls of $I_1$ and $I_2$ are disjoint. Then
  $a\notin A$ if and only if there exists an $H\in\sH$ containing
  $\conv I_1$ which intersects neither $\conv I_2$ nor $a$, that is,
  does not intersect $\conv(\{a\}\cup I_2)$. By the separation
  theorem, this is possible if and only if $\conv I_1$ and
  $\conv(\{a\}\cup I_2)$ are disjoint.  Hence, $a\in\partial A$ means
  that $\conv I_1$ touches $\conv(\{a\}\cup I_2)$, and thus there is a
  face $F_1$ of $\conv I_1$ and a face $F_2$ of $\conv(\{a\}\cup I_2)$
  of dimensions $d_1,d_2\le d-1$, respectively, such that $F_1$ and
  $F_2$ lie on an affine subspace $E$ of dimension at most $d-1$ and
  $F_1\cap F_2\ne\emptyset$. Moreover, it follows from the assumption
  $\conv I_1\hspace{.5pt}\cap\hspace{.5pt}\conv I_2=\emptyset$ that
  $F_2$ has $a$ as one of its vertices.
  
  Let $x\in F_1\cap F_2$. Then
  \begin{displaymath}
    \sum_{k=1}^{d_1} s_kx_k=x=t_0a+\sum_{k=1}^{d_2} t_ky_k,
  \end{displaymath}
  where $x_1,\ldots,x_{d_1}$ are vertices of $F_1$ and thus belong to
  $I_1$, $y_1,\ldots,y_{d_2}$ are vertices of $F_2$ other than $a$ and
  thus belong to $I_2$, $s_1,\ldots,s_{d_1}$ and $t_0,\ldots,t_{d_2}$
  are positive numbers which sum up to one in each group. Hence,
  \begin{displaymath}
    a=\sum_{k=1}^{d_1}\frac{s_k}{t_0}x_k+\sum_{k=1}^{d_2}
    \Bigl(-\frac{t_k}{t_0}\Bigr)y_k.
  \end{displaymath}
  Since $x_1,\ldots,x_{d_1},y_1,\ldots,y_{d_2}$ belong to $E$ and the
  dimension of $E$ is at most $d-1$, it is possible to extract from
  them at most $d$ points which also generate $E$ and yield $a\in E$ as
  their affine combination. 
\end{proof}

\section{Strong laws of large numbers}
\label{sec:SLLN}

The following result may be regarded as the Marcinkiewicz--Zygmund
strong law of large numbers (see \cite[Theorem~6.3.2]{MR2977961}) for
$\emp(\alpha)$ in the inclusion sense.

\begin{theorem}
  \label{th:MZSLLN}
  For $p\in[1,2)$, let $\E|\xi_1|^p<\infty$.
  Then, for any $\eps>0$, 
  \begin{displaymath}
    \sR\Big(\alpha +n^\frac{1-p}{p}\eps\Big)
    \subset
    \emp(\alpha)
    \subset
    \sR\Big(\alpha-n^\frac{1-p}{p}\eps\Big), \quad \alpha\in(0,1],
  \end{displaymath}
%  almost surely for all sufficiently large $n$.
  for all $n\geq n_0$ with a finite (possibly, random) $n_0$.
\end{theorem}
	
We preface the proof of Theorem~\ref{th:MZSLLN} with the following
lemma. Define
\begin{equation*}
  D_n=\sup\big\{\bigl|\emm(H)-\mu(H)\bigr|\colon H\in\sH_0\big\},\quad  n\ge 1,
\end{equation*}
where $\sH_0$ is defined in Lemma~\ref{th:count}.  
	
\begin{lemma}
  \label{lem:MZSLLN}
  Under conditions of Theorem \ref{th:MZSLLN},  
  \begin{equation}
    \label{eq:MZSLLN}
    n^{\frac{p-1}p}D_n\to 0\quad \text{ a.s. as }\; n\to\infty.
  \end{equation}
\end{lemma}
\begin{proof}
  First, recall the definition of bracketing number
  and entropy with bracketing. Let $(\sX,\sA)$ be a measurable space,
  and let $\sF$ be a set of measurable functions $f\colon\sX\to\R$.
  For a measure $\theta$ on $(\sX,\sA)$, let $\|f\|_p$, $p\ge1$, be
  the $\Lp(\theta)$-norm of a function $f$.

  The \emph{bracketing number}
  $N_{[\,]}\bigl(\eps,\sF,\Lp(\theta)\bigr)$ is the minimal number of
  \emph{$\eps$-brackets}
  \begin{displaymath}
    B_{l,u}=\{f\colon l(x)\le f(x)\le u(x),x\in\sX\}
  \end{displaymath}
  with $l,u\in\Lp(\theta)$ and $\|u-l\|_p<\eps$, needed to cover the
  set $\sF$.  The \emph{entropy with bracketing}
  $H_{[\,]}\bigl(\eps,\sF,\Lp(\theta)\bigr)$ is defined as
  $\log N_{[\,]}\bigl(\eps,\sF,\Lp(\theta)\bigr)$, see
  Definition~2.1.6 in \cite{VaartWellner}.
  
  Denote by $\nu$ the law of $\xi_1$ and by $\theta=\nu\otimes\mu$ the
  law of $(\xi_1,X_1)$. Then the empirical measure corresponding to
  $\theta$ is
  \begin{equation}
    \label{eq:theta-n}
    \empm=n^{-1}\sum_{i=1}^n\delta_{(\xi_i,X_i)}.
  \end{equation}
  Denote by
  $\sF_0$ the family of functions $f_{H}\colon\R\times\Rd\to\R$,
  $H\in\sH_0$, given by $f_{H}(t,x)=t\one_{H}(x)$.  Since
  \begin{equation}
    \label{eq:mutheta}
    \emm(H)=\int_{\R\times\Rd}f_{H}\,\mathrm{d}\empm,\qquad
    \mu(H)=\int_{\R\times\Rd}f_{H}\,\mathrm{d}\theta,
  \end{equation}
  we can write
  \begin{displaymath}
    D_n=\sup_{f\in\sF_0}\left|\int_{\R\times\Rd}f\,\mathrm{d}\empm-
      \int_{\R\times\Rd}f\,\mathrm{d}\theta \right|.
  \end{displaymath}
  
  It follows from Corollary 6 in \cite{dutta22} that $\sH$ can be
  covered by $n_\delta=(2/\delta)^{c_d}$ $\delta$-brackets with some
  $c_d=O(d)$: for each $\delta>0$ there exists a finite family
  $(A_i^1,A_i^2)$, $i=1,\ldots,n_\delta$, of Borel sets in $\Rd$ with
  $\mu(A_i^2\setminus A_i^1)<\delta$ such that each $H\in\sH$
  satisfies $A_i^1\subset H\subset A_i^2$ for some $i$. Then, the
  functions
  \begin{align*}
    l_i=\min\{f_{A_i^1},f_{A_i^2}\},\quad  
    u_i=\max\{f_{A_i^1},f_{A_i^2}\}
  \end{align*}
  with $f_{A_i^j}(t,x)=t\one_{A_i^j}(x)$, $i=1,\ldots,n_\delta$,
  $j=1,2$, form a family of $\eps$-brackets for $\sF_0$ with
  $\eps=\|\xi_1\|_p\,\delta^{1/p}$. This follows from
  $l_i\le f_{H}\le u_i$ and
  \begin{displaymath}
    \|u_i-l_i\|_p^p=\|f_{A_i^2}-f_{A_i^1}\|_p^p
    =\int_{\R}|t|^p\,\mathrm d\nu\;\int_{\Rd}
    \one_{A_i^2\setminus A_i^1}(x)\,\mathrm d\mu
    =\E|\xi_1|^p\;\mu(A_i^2\setminus A_i^1)<\|\xi_1\|_p^p\,\delta.
  \end{displaymath}
  Hence,
  \begin{displaymath}
    H_{[\,]}\bigl(\eps,\sF_0,\Lp(\theta)\bigr)\le\log(2/\delta)^{c_d}\sim
    pc_d\log(1/\eps)\quad \text{as}\; \eps\downarrow0,
  \end{displaymath}
  and
  \begin{displaymath}
    \int_0^1
    \bigl(H_{[\,]}\bigl(\eps,\sF_0,\Lp(\theta)\bigr)\bigr)^{1-1/p}\,\mathrm
    d\eps<\infty.
  \end{displaymath}
  
  Since $\sH_0$ is countable and the envelope
  $F(t,x)=\sup_{f\in\sF_0}|f(t,x)|=|t|$ belongs to $\Lp(\theta)$ due
  to $\int_{\R\times\Rd}|t|^p\,\mathrm{d}\theta=\E|\xi_1|^p<\infty$,
  we may apply Theorem~2.4.1 in \cite{VaartWellner} if $p=1$ and
  Theorem~1 in \cite{mas:rio98} if $p\in(1,2)$ to obtain
  \eqref{eq:MZSLLN}.
\end{proof}

\begin{proof}[Proof of Theorem \ref{th:MZSLLN}]
  Fix any $\eps>0$. Lemma \ref{lem:MZSLLN} ensures that
  $D_n<\frac\eps2 n^{\frac{1-p}p}$ a.s.~for large $n$. By the
  Marcinkiewicz--Zygmund strong law of large numbers for
  $\{\xi_i,i\ge1\}$, we also have
  $|\bar\xi_n-1|<\frac\eps2 n^{\frac{1-p}p}$ a.s.~for large $n$.
  Hence, for such $n$, we have
  \[1-\bigl(\alpha+\eps n^{\frac{1-p}p}\bigr)<\bar\xi_n-\alpha-\frac\eps2 n^{\frac{1-p}p}<\bar\xi_n-\alpha+\frac\eps2 n^{\frac{1-p}p}<1-\bigl(\alpha-\eps n^{\frac{1-p}p}\bigr).\]
  
  Therefore, for large $n$, the following inclusions hold:
  \begin{align*}
    \cap \bigl\{H\in\sH_0& \colon \mu(H)>1-(\alpha+\eps
    n^{\frac{1-p}p})\bigr\}
     \:\subset\:
     \cap \bigl\{H\in\sH_0\colon \mu(H)>\bar\xi_n-\alpha-\eps n^{\frac{1-p}p}/2\bigr\}\\
    &\subset\;
    \cap \bigl\{H\in\sH_0\colon \emm(H)>\bar\xi_n-\alpha \bigr\}
    \;\subset\;
    \cap \bigl\{H\in\sH_0\colon \mu(H)>\bar\xi_n-\alpha+\eps n^{\frac{1-p}p}/2\bigr\}\\
    &\subset\;
      \cap \bigl\{H\in\sH_0\colon \mu(H)>1-(\alpha-\eps n^{\frac{1-p}p}) \bigr\}.
  \end{align*}
  The result now follows from Lemma~\ref{th:count}.
\end{proof}

\section{Law of the iterated logarithm}
\label{sec:LIL}

In this section, we consider the law of the iterated logarithm for
empirical depth trimmed regions $\emp(\alpha)$ in the inclusion sense.
Define the numerical sequence
\begin{equation*}
  \label{eq:loglog}
  \lambda_n=\sqrt{2n^{-1}\log\log n}, \quad n\geq 3.
\end{equation*}

\begin{theorem}
  \label{th:LIL}
  Assume that $M=\E\xi_1^2<\infty$.
  Then
  \begin{enumerate}
  \item[(i)] for all $\alpha\in(0,1]$ and any
    $\gamma>\sqrt{M\alpha-\alpha^2}$ 
    \begin{displaymath}
      \sR\bigl(\alpha+\gamma\lambda_n\bigr)
      \subset\emp(\alpha)\subset
      \sR\bigl(\alpha-\gamma\lambda_n\bigr)
    \end{displaymath}
    almost surely for all sufficiently large $n$;
  \item[(ii)] under (C1), (C2), and for all $\alpha\in(0,\alpha_{\max})$,
    for any $\gamma<\sqrt{M\alpha-\alpha^2}$ almost surely there exist
    sequences $n'_k,n''_k\to\infty$ such that
    \begin{displaymath}
      \sR\bigl(\alpha+\gamma\lambda_{n'_k}\bigr)
      \not\subset\emp[{n'_k}](\alpha)\quad\text{and}
      \quad\emp[{n''_k}](\alpha)\not\subset
      \sR\bigl(\alpha-\gamma\lambda_{n''_k}\bigr);
    \end{displaymath}
  \item[(iii)] for any $\gamma>0$ almost surely there exists a sequence
    $n'''_k\to\infty$ such that
    \begin{displaymath}
      \sR\bigl(\alpha+\gamma\lambda_{n'''_k}\bigr)
      \subset\emp[{n'''_k}](\alpha)\subset
      \sR\bigl(\alpha-\gamma\lambda_{n'''_k}\bigr).
    \end{displaymath}
  \end{enumerate}
\end{theorem}

We will need some preliminary lemmas.  As in the proof of
Lemma~\ref{lem:MZSLLN}, denote by $\nu$ the law of $\xi_1$ and by
$\theta=\nu\otimes\mu$ the law of $(\xi_1,X_1)$.  Let
$l^{\infty}(\sH)$ be the space of all bounded functions
$h\colon\sH\to\R$ with
\begin{displaymath}
  \|h\|_{\sH}=\sup_{H\in\sH}|h(H)|<\infty.
\end{displaymath}
	
\begin{lemma}
  \label{lem:LIL}
  Let
  \begin{multline}
    \label{eq:K}
    \mathbb{K}=
    \Bigl\{
    h_w\colon H\mapsto\int_{\R\times H}t w(t,x)\,\mathrm{d}\theta,\\
    \text{ where } w\in L^2(\theta),\;
    \int_{\R\times\Rd}w(t,x)\,\mathrm{d}\theta=0,\;
    \int_{\R\times\Rd}w^2(t,x)\,\mathrm{d}\theta\leq 1		
    \Bigr\rbrace,		
  \end{multline}
  which is a compact subset of $l^{\infty}(\sH)$.  Then
  \begin{displaymath}
    \inf_{h_w\in\mathbb{K}}
    \left\|\frac{\emm-\mu}{\lambda_{n}}-h_w \right\|_{\sH}
    =\inf_{h_w\in\mathbb{K}}\sup_{H\in\sH}
    \left|\frac{\emm(H)-\mu(H)}{\lambda_{n}}-h_w(H)\right|\to 0
    \quad \text{a.s. as} \; n\to\infty.
  \end{displaymath}
  Moreover, almost surely, the set of limit points of
  $\bigl\{(\emm-\mu)/\lambda_{n},n\ge1\bigr\}$ in $l^\infty(\sH)$ is
  exactly $\mathbb{K}$.
\end{lemma}
\begin{proof}
  Denote by $\sF$ the set of functions $f_{H}(t,x)=t\one_{H}(x)$ with
  $H\in\sH$, by $F(t,x)=|t|$ its envelope, and by $\empm$ the
  empirical measure defined at \eqref{eq:theta-n}.  The set $\sF$ is a
  bounded countable determined subset of $L^2(\theta)$ (see page~158
  in \cite{alex:tal89}), which is also a Vapnik--\v{C}ervonenkis (or
  VC) graph class. The latter follows from the VC-property of the set
  of halfspaces in $\Rd$ and Lemma~2.6.18(vi) in \cite{VaartWellner}.
	
  Moreover, $\sF$ is totally bounded with respect to the centred
  $\Lp[2](\theta)$-pseudometric given by
  \begin{displaymath}
    \rho_o\bigl(f_{H_1},f_{H_2}\bigr)=\biggl(\int_{\R\times\Rd}
    \bigl(f_{H_2}^o-f_{H_1}^o\bigr)^2
    \,\mathrm{d}\theta\biggr)^\frac{1}{2},\quad H_1,H_2\in\sH,
  \end{displaymath}
  with
  \begin{displaymath}
    f_{H_j}^o(t,x)=f_{H_j}(t,x)-\int_{\R\times\Rd}f_{H_j}\,\mathrm d\theta=
    f_{H_j}(t,x)-\mu(H_j), \quad j=1,2.
  \end{displaymath}
  Indeed, for all $H_1,H_2\in\sH$,
  % \begin{displaymath}
  %   \rho\bigl(H_1,H_2\bigr)
  %   =,\quad H_1,H_2\in\sH,
  % \end{displaymath}
  % in the following way:
  \begin{align*}
    \rho_o^2\bigl(f_{H_1},f_{H_2}\bigr)
    &=\int_{\R\times\Rd}\Bigl(\bigl(
      f_{H_2}(t,x)-f_{H_1}(t,x)\bigr)-\bigl(\mu(H_2)-\mu(H_1)\bigr)\Bigr)^2
      \,\mathrm d\theta\\&\le 2\int_{\R\times\Rd}
    \bigl(f_{H_2}(t,x)-f_{H_1}(t,x)\bigr)^2\,\mathrm d\theta+
	2\int_{\R\times\Rd}
    \bigl(\mu(H_2)-\mu(H_1)\bigr)^2\,\mathrm d\theta\\
    &\le
      2\,\E\xi_1^2\,\mu\bigl(H_1\Delta H_2\bigr)
      +2\,\mu\bigl(H_1\Delta H_2\bigr)=2(M+1)\,
      \mu\bigl(H_1\Delta H_2\bigr),
  \end{align*}
  so that $\rho_o$ is bounded by the symmetric difference pseudometric
  on $\sH$.  Hence, the total boundedness of $\sF$ with respect to
  $\rho_o$ follows from that of $\sH$ with respect to the symmetric
  difference pseudometric. The latter holds due to finiteness of
  bracketing numbers, see the proof of Lemma~\ref{lem:MZSLLN}.
  Finally, the envelope function $F$ satisfies
  \begin{displaymath}
    \E\left[\frac{F^2(\xi_1,X_1)}{\log\log F(\xi_1,X_1)}\right]
    \le\E\xi_1^2<\infty\,.
  \end{displaymath}
  Thus, the set $\sF$ fulfils all assumptions of Theorem~2.3 in
  \cite{alex:tal89}. Hence, for empirical measures $\empm$, the
  compact law of the iterated logarithm holds with the limiting set,
  which consists of all functions $h_{w}\colon f_{H}\mapsto
  \int_{\R\times\Rd}f_{H}(t,x)w(t,x)\,\mathrm{d}\theta$, where $w$
  satisfies conditions on the right-hand side of \eqref{eq:K}.

  Both claims are confirmed, since $f_{H}(t,x)=t\one_{H}(x)$ and, by
  \eqref{eq:mutheta},
  \begin{displaymath}
    \int_{\R\times\Rd}f_{H}\,\mathrm
    d\empm-\int_{\R\times\Rd}f_{H}\,\mathrm d\theta=\emm(H)-\mu(H). \qedhere
  \end{displaymath}
\end{proof}

\begin{lemma} \label{lem:supK}
  For any $H\in\sH$,
  \begin{gather}
    \label{eq:supK}
    \sup_{h_w\in\mathbb K}h_w(H)=\sqrt{M\mu(H)-\mu^2(H)},\\
    \inf_{h_w\in\mathbb K}h_w(H)=-\sqrt{M\mu(H)-\mu^2(H)}.
    \label{eq:infK}
  \end{gather}
\end{lemma}
\begin{proof}
  Since, by \eqref{eq:K}, the integral of $w$ vanishes,   
  \begin{displaymath}
    h_w(H)=
    \int_{\R\times\Rd}t\one_{H}(x)w(t,x)\,\mathrm d\theta=
    \int_{\R\times\Rd}\bigl(t\one_{H}(x)-\mu(H)\bigr)w(t,x)
    \mathrm d\theta
  \end{displaymath}
  for all $h_w\in\mathbb K$. Hence, by the Cauchy--Schwarz inequality,
  \begin{align*}
    h_w^2(H)
    &\le\int_{\R\times\Rd}
      \bigl(t\one_{H}(x)-\mu(H)\bigr)^2\,\mathrm d\theta
      \int_{\R\times\Rd} w^2(t,x)\,\mathrm d\theta\\
    &\le\operatorname{Var}\bigl(\xi_1\one_{H}(X_1)\bigr)=M\mu(H)-\mu^2(H).
  \end{align*}
  It is straightforwardly checked that both inequalities above turn
  into equalities for
  \begin{equation}
    \label{eq:extf}
    w_{H}(t,x)=\frac{t\one_{H}(x)-\mu(H)}{\sqrt{M\mu(H)-\mu^2(H)}},
  \end{equation}
  and $h_{w_{H}}$ belongs to $\mathbb K$. Thus, \eqref{eq:supK}
  follows. The proof of \eqref{eq:infK} is similar.
\end{proof}

\begin{proof}[Proof of Theorem \ref{th:LIL}]
  We begin with the right-hand inclusion in (i). Assume the opposite,
  i.e., that there are $n_k\to\infty$ and $x_k\in\Rd$ such that
  $x_k\in\emp[n_k](\alpha)$ and
  $x_k\notin\sR(\alpha-\gamma\lambda_{n_k})$ for all $k\ge1$. By
  definition of the halfspace depth trimmed regions, 
  $\emm[n_k](H_k)\ge\alpha$ and $\mu(H_k)<\alpha-\gamma\lambda_{n_k}$
  for any $k\ge1$ and some $H_k\in\sH$.  Moreover, by Lemma
  \ref{lem:MZSLLN}, for any $\eps>0$, we have 
  $\mu(H_k)>\alpha-\gamma\lambda_{n_k}-\eps$ for all sufficiently
  large $k$. Hence, for such $k$,
  \[\frac{\emm[n_k](H_k)-\mu(H_k)}{\lambda_{n_k}}>\gamma,\]
  whence
  \begin{equation}\label{eq:bnd1}
    \sup_{\substack{H\in\sH\colon\\\alpha-\gamma\lambda_{n_k}-\eps<\mu(H)<
	\alpha-\gamma\lambda_{n_k}}}\hspace{-25pt}
    \frac{\emm[n_k](H)-\mu(H)}{\lambda_{n_k}}>\gamma.
  \end{equation}

  On the other hand, Lemma \ref{lem:LIL} for each $\delta>0$ and
  $k\ge k'_0(\delta)$ imply that
  \begin{equation}\label{eq:bnd2}
    \sup_{\substack{H\in\sH\colon\\\alpha-\gamma\lambda_{n_k}-\eps<\mu(H)<
	\alpha-\gamma\lambda_{n_k}}}\hspace{-25pt}
    \frac{\emm[n_k](H)-\mu(H)}{\lambda_{n_k}}<
    \delta+\hspace{-25pt}\sup_{\substack{H\in\sH\colon\\
	\alpha-\gamma\lambda_{n_k}-\eps<\mu(H)<\alpha-\gamma\lambda_{n_k}}}
    \hspace{-25pt}\sup_{h_w\in\mathbb K}h_w(H).
  \end{equation}
  By Lemma~\ref{lem:supK}, the right-hand side for large $k$ does not
  exceed $\delta+\sup_{\alpha-2\eps<x<\alpha}\sqrt{Mx-x^2}$. Letting
  first $\eps\downarrow0$ and then $\delta\downarrow0$, we obtain from
  \eqref{eq:bnd1} and \eqref{eq:bnd2} that
  $\gamma\le\sqrt{M\alpha-\alpha^2}$, which is a contradiction with
  the condition imposed in (i).  The left-hand inclusion in (i) is
  proved in the same way.

  Let us now turn to the proof of the left-hand non-inclusion in (ii).
  Consider an arbitrary point $x\in\partial\sR(\alpha)$. By
  Lemma~\ref{lem:level}, $D(x)=\alpha$, and, by (C1), there exists a
  halfspace $H$ supporting $\sR(\alpha)$ at $x\in\partial H$ and with
  $\mu(H)=\alpha$.  Then, by Lemma~\ref{lem:LIL}, there exists a
  sequence $n'_k\to\infty$ such that
  \begin{equation}
    \label{eq:def_n'_k}
    \frac{\emm[n'_k]-\mu}{\lambda_{n'_k}}\to-h_{w_{H}}\quad 
    \text{ a.s.~in }l^\infty(\sH)\;\text{ as }\; k\to\infty,
  \end{equation}
  where the extremal function $h_{w_{H}}$ is defined by
  \eqref{eq:extf} and \eqref{eq:K}. Written explicitly,
  \begin{equation}
    \label{eq:hwex}
    h_{w_{H}}(G)=\frac{M\mu(G\cap H)
      -\alpha\mu(G)}{\sqrt{M\alpha-\alpha^2}},\quad G\in\sH.
  \end{equation}
  For $k\ge1$, let $H_k$ denote a translate of $H$ such that
  $\partial H_k$ supports $\partial\sR(\alpha+\gamma\lambda_{n'_k})$
  which is eventually not empty due to $\alpha<\alpha_{\max}$. Choose
  any point
  \begin{displaymath}
    x_k\in\partial H_k\cap\partial\sR(\alpha+\gamma\lambda_{n'_k}).
  \end{displaymath}
  By construction and due to Lemma~\ref{lem:level}(ii),
  $H_k$ is a minimal halfspace at $x_k$, and thus
  $\mu(H_k)=\alpha+\gamma\lambda_{n'_k}$.

  Since the convergence in \eqref{eq:def_n'_k} is uniform over $\sH$
  and by \eqref{eq:hwex},
  \begin{equation}
    \label{eq:mutoh}
    \lim_{k\to\infty}\frac{\emm[n'_k](H_k)-\mu(H_k)}{\lambda_{n'_k}}
    =-\lim_{k\to\infty}h_{w_{H}}(H_k)=
    -\lim_{k\to\infty}\frac{M\mu(H_k\cap H)-
      \alpha\mu(H_k)}{\sqrt{M\alpha-\alpha^2}},
  \end{equation}
  where $\mu(H_k\cap H)=\mu(H)=\alpha$ and $\mu(H_k)\to\alpha$. So,
  the right-hand side in \eqref{eq:mutoh} becomes
  $-\sqrt{M\alpha-\alpha^2}$.  Therefore, for any
  $\gamma<\sqrt{M\alpha-\alpha^2}$ and sufficiently large $k$, 
  \begin{displaymath}
    \frac{\emm[n'_k](H_k)-\mu(H_k)}{\lambda_{n'_k}}<-\gamma,
  \end{displaymath}
  and thus $\emm[n'_k](H_k)<\mu(H_k)-\gamma\lambda_{n'_k}=\alpha$. 
  Hence, $x_k\notin\emp[n'_k](\alpha)$ but
  $x_k\in\sR(\alpha+\gamma\lambda_{n'_k})$ which proves the claim.

  For the proof of the right-hand non-inclusion in (ii), consider a
  point $x\in\partial\sR(\alpha)$ and assume additionally that
  $\sR(\alpha)$ has a unique supporting hyperplane at $x$. The
  existence of such points follows from Theorem~2.2.5 and the argument
  thereafter in \cite{schn2}. By
  Lemma~\ref{lem:level}(ii), this means that the minimal
  halfspace $H$ at such $x$ is unique. As above, there are
  $n''_k\to\infty$ such that
  \begin{equation}
    \label{eq:def_n''_k}
    \frac{\emm[n''_k]-\mu}{\lambda_{n''_k}}\to h_{w_{H}}\quad
    \text{ a.s.~in }l^\infty(\sH)\; \text{ as }\; k\to\infty,
  \end{equation}
  where $h_{w_{H}}$ is given by \eqref{eq:hwex}. For $k\ge1$, consider
  some $x_k\in\partial\sR(\alpha-\gamma\lambda_{n''_k})$ such that
  $x_k\to x$. The existence of such sequence is guaranteed by the
  continuity of $\sR(\cdot)$ at $\alpha$ in the Hausdorff metric which
  follows from (C1), (C2), and $\alpha<\alpha_{\max}$.

  We now prove that, for large $k$ and any $G_k\in\sH$ with
  $x_k\in\partial G_k$, we have $\emm[n''_k](G_k)\ge\alpha$. Assume
  the contrary, i.e., there is a subsequence $k_l$ violating this
  condition, which without loss of generality is assumed to be the
  original sequence, so that 
  $G_{k}\in\sH$,  $x_{k}\in\partial G_{k}$ and 
  $\emm[n''_{k}](G_{k})<\alpha$. Denoting by $u_{k}\in\Sphere$
  the outer normal to $G_{k}$, by compactness of $\Sphere$ and
  passing to a subsequence, assume that  $u_{k}\to u\in\Sphere$.

  Let $G=H_u+x$. Since $u_{k}\to u$ and $x_{k}\to x$, (C1) implies
  that $\mu(G_{k})\to\mu(G)$ and $\mu(G_{k}\cap H)\to\mu(G\cap H)$ as
  $k\to\infty$.  By \eqref{eq:def_n''_k} and \eqref{eq:hwex}, 
  \begin{equation}
    \label{eq:limG^-}
    \lim_{l\to\infty}\frac{\emm[n''_{k}](G_{k})-\mu(G_{k})}
    {\lambda_{n''_{k}}}=\lim_{l\to\infty}h_{w_{H}}(G_{k})=h_{w_{H}}(G).
  \end{equation}
  If $\mu(G)>\alpha$ then $\mu(G_{k})$ is at least $\alpha+\eps$ for
  some $\eps>0$ and all sufficiently large $k$, and the left-hand side
  in \eqref{eq:limG^-} becomes $-\infty$, giving a
  contradiction. Thus, $\mu(G)=\alpha=\mu(H)$, and, by uniqueness of
  the minimal halfspace at $x$, $G=H$. Hence, by \eqref{eq:limG^-},
  \begin{equation*}
    \lim_{l\to\infty}\frac{\emm[n''_{k}](G_{k})-\mu(G_{k})}
    {\lambda_{n''_{k}}}=h_{w_{H}}(H)=\sqrt{M\alpha-\alpha^2}>\gamma.
  \end{equation*}
  Since
  $x_{k}\in\partial\sR(\alpha-\gamma\lambda_{n''_{k}})\cap\partial
  G_{k}$, for large $l$,
  \[\emm[n''_{k}](G_{k})>\mu(G_{k})+\gamma\lambda_{n''_{k}}\ge\alpha,\]
  which is a contradiction.

  We have proved that $\emm[n''_k](G_k)\ge\alpha$ for large $k$
  and any $G_k$ which contains $x_k$. Hence,
  $x_k\in\partial\sR(\alpha-\gamma\lambda_{n''_k})\cap\emp[n''_k](\alpha)$
  for large $k$. Moreover, almost surely
  $x_k\notin\partial\emp[n''_k](\alpha)$ for all $k$, since the latter
  boundary is contained in the union of all affine subspaces generated by
  $i$-tuples of points from $\{X_1,\dots,X_n\}$ with $i=1,\ldots,d$, see
  Lemma~\ref{lemma:boundary}, and $\mu$ satisfies (C1). Hence, for
  each large $k$, there is an $\tilde x_k$ from a neighbourhood of
  $x_k$ with
  $\tilde
  x_k\in\emp[n''_k](\alpha)\setminus\sR(\alpha-\gamma\lambda_{n''_k})$
  which was to be proved.

  Let us now prove (iii). By Lemma \ref{lem:LIL}, for any $\gamma>0$
  there exist $n'''_k\to\infty$ such that
  \begin{displaymath}
    \frac{\emm[n'''_k]-\mu}{\lambda_{n'''_k}}\to0\in\mathbb K
    \quad \text{ a.s.~in } l^\infty(\sH).
  \end{displaymath}
  Hence, for any $\gamma>0$ and sufficiently large $k$,
  \begin{equation}
    \label{eq:ineqiii}
    -\gamma<\inf_{H\in\sH}\frac{\emm[n'''_k](H)-\mu(H)}
    {\lambda_{n'''_k}}\le\sup_{H\in\sH}\frac{\emm[n'''_k](H)-\mu(H)}
    {\lambda_{n'''_k}}<\gamma.
  \end{equation}
  If the right-hand inclusion in (iii) were not satisfied, then, as in
  the proof of (i), we would have $\emm[n'''_k](H_k)\ge\alpha$ and
  $\mu(H_k)<\alpha-\gamma\lambda_{n'''_k}$ for some $H_k\in\sH$, which
  contradicts the right-hand inequality in \eqref{eq:ineqiii}. The
  proof of the left-hand inclusion is similar.
\end{proof}

\section{Distance between depth trimmed regions}
\label{sec:diff-depth-pert}

In order to establish the rate of convergence of $\emp(\alpha)$ to $\sR(\alpha)$
in the Hausdorff metric, we first need to study the asymptotic
behaviour of the Hausdorff distance
$\rho_H\bigl(\sR(\alpha),\sR(\beta)\bigr)$ as $\beta\to\alpha$.
For this, we need a further smoothness condition on the depth function
$D$. In the following, we fix some $\alpha\in(0,\alpha_{\max})$.
\begin{enumerate}
\item[(C3)] The depth function $D$ is continuously differentiable in a
  neighbourhood of $\partial\sR(\alpha)$ with $\grad D(x)\ne0$ for
  all $x\in\partial\sR(\alpha)$.
\end{enumerate}

Condition (C3) is generally difficult to check without calculating the
depth function explicitly; we show in Section~\ref{sec:appendix} that
it holds for uniform distributions on a sufficiently wide family of
convex bodies.
While,  in general,  a minimal  halfspace  at $x$  is not  necessarily
unique,  the  differentiability  of $D$  at  $x\in\partial\sR(\alpha)$
imposed in (C3)  guarantees the uniqueness. Indeed, let $u_x$ denote the
unit outer normal to a minimal halfspace at $x$, so that
\begin{displaymath}
  \mu(x+\Hx)=\min_{H\in\sH,\,H\ni x}\mu(H)=\alpha. 
\end{displaymath}
Since, by Lemma \ref{lem:level}, $x+\partial \Hx$ supports
$\partial\sR(\alpha)$, which is the $\alpha$-level surface of $D$,
this hyperplane is orthogonal to $\grad D(x)$. Hence,
\begin{equation}
  \label{eq:outnorm}
  u_x=\frac{\grad D(x)}{\|\grad D(x)\|}.
\end{equation}

For each $u\in\Sphere$, let $\mu_u$ be the measure on $\R$, obtained
as the projection of $\mu$ on the line $\{tu\colon t\in\R\}$.

\begin{enumerate}
\item[(C4)] For any $u$, the measure $\mu_u$ is absolutely continuous
    and admits a density $f_u(t)$, which is continuous in $(t,u)$ and does
    not vanish for all $t$ from the open interval between the
    essential infimum and essential supremum of the support of
    $\mu_u$.
\end{enumerate}

Denote
\begin{displaymath}
  T_{x,u}\mu= f_u(\langle x,u\rangle).
\end{displaymath}
If $\mu$ is absolutely continuous with density $f$, then 
\begin{displaymath}
  T_{x,u}\mu=\int_{u^\perp}f(x+y)\,\mathrm dy,
  \qquad u\in\Sphere,\,x\in\Rd,
\end{displaymath}
where $u^\perp=\{y\in\Rd\colon\langle u,y\rangle=0\}$, so that
$T_{x,u}\mu$ is the \emph{Radon transform} of $f$.

\begin{remark}
  \label{rem:min}
  Note that, due to (C3), $u_x$ defined by \eqref{eq:outnorm} is
  continuous in $x$. Hence, by (C4), $T_{x,u_x}\mu$ is also continuous
  in $x$, and $\min_{x\in\partial\sR(\alpha)}T_{x,u_x}\mu$ is well
  defined and positive. Furthermore, (C4) implies the validity of (C1)
  and (C2).
\end{remark}

For a function $\phi\colon\R^d\to\R$, define a perturbed depth trimmed
region
\begin{displaymath}
  \sR(\alpha; \phi)=\cl\bigl\{x\colon D(x)\geq \alpha+\phi(x)\bigr\}.
\end{displaymath}
Furthermore, let 
\begin{equation}
  \label{eq:Phi}
  \Phi(\phi)=\rho_H\bigl(\sR(\alpha),\sR(\alpha;\phi)\bigr).
\end{equation}
This functional was studied in \cite{mol98} for a general function $D$
and sets $\sR(\alpha)$ and $\sR(\alpha;\phi)$ defined with the
opposite inequality sign.

Following Section~8 in \cite{borovkov}, the functional $\Phi$ is said
to be continuously differentiable if there exists a functional $\Phi'$
defined on continuous functions $\phi\colon\R^d\to\R$ such that
\begin{equation}
  \label{eq:Phi'}
  \delta_n^{-1}\Phi(\delta_n\phi_n)
  \to \Phi'(\phi)\quad
  \text{and}\quad \Phi'(\phi_n)\to \Phi'(\phi)\quad 
  \text{as}\;
  n\to\infty
\end{equation}
for $\delta_n\downarrow0$ and any sequence $\{\phi_n,n\ge1\}$ of (not
necessarily continuous) functions which converge to a continuous
function $\phi$ uniformly on compact sets.

Following \cite{mol98}, define
\begin{gather}
  \label{eq:varpi_+}
  \varpi_+(x,r)=\sup_{y\in B_r(x)}\bigl(D(y)-D(x)\bigr),\quad x\in\Rd,\\
  \varpi_-(x,r)=\inf_{y\in B_r(x)}\bigl(D(y)-D(x)\bigr),\quad x\in\Rd,
  \label{eq:varpi_-}
\end{gather}
where $B_r(x)$ is the closed Euclidean ball of radius $r$ centred at
$x$. Clearly, $\varpi_+$ and $\varpi_-$ are continuous in both
arguments due to continuity of $D$ which follows from (C1).  Moreover,
there is a neighbourhood $\sG(\alpha)$ of $\partial\sR(\alpha)$ such
that $\varpi_+$ and $\varpi_-$ are differentiable in $r$ at $r=0$
uniformly over $x\in\sG(\alpha)$, which is confirmed by the following
lemma.

\begin{lemma}
  \label{lem:varpi}
  Assume that (C3) and (C4) hold. Then, for each $\alpha\in(0,\alpha_{\max})$,
  \begin{displaymath}
    \lim_{r\downarrow0}\sup_{x\in\sG(\alpha)}
    \Bigl|\frac{\varpi_+(x,r)}r-T_{x,u_x}\mu\Bigr|=0,\qquad
    \lim_{r\downarrow0}\sup_{x\in\sG(\alpha)}
    \Bigl|\frac{\varpi_-(x,r)}r+T_{x,u_x}\mu\Bigr|=0.
  \end{displaymath}
  Furthermore, $\grad D(x)=(T_{x,u_x}\mu)u_x$.
\end{lemma}

Note that, for compactly supported measures $\mu$ and except for
uniformity in $x\in\sG(\alpha)$, a related result appears as
equation~(2.7) in \cite{molina-fructuoso22}.

\begin{proof}[Proof of Lemma \ref{lem:varpi}]
  We will prove the first equality, the argument for the second is
  similar.  Denote by $\sG(\alpha)$ a sufficiently small open
  neighbourhood of $\partial\sR(\alpha)$. By (C3) and (C4), we may
  assume that $\grad D(x)$ and $T_{x,u}\mu$ are uniformly continuous in
  $\sG(\alpha)$ and $\sG(\alpha)\times\Sphere$, respectively.
	
  By \eqref{eq:varpi_+}, for small $r>0$ we have
  \begin{align}
    \label{eq:varpi-sup}
    \varpi_+(x,r)
    &=\sup_{y\in
      B_r(x)}\bigl(\mu(y+\Hx[y])-\mu(x+\Hx)\bigr)\notag \\
    &\le\sup_{y\in B_r(x)}\bigl(\mu(y+\Hx[x])-\mu(x+\Hx)\bigr)
      =\mu(x+ru_x+\Hx)-\mu(x+\Hx).
  \end{align}
  In the second line, the inequality is due to minimality of
  $y+\Hx[y]$ at $y$, and the ultimate equality holds since the
  supremum is attained at the point of $B_r(x)$ farthest from $x+\Hx$.
  By definition of $T_{x,u}\mu$, the right-hand side of
  \eqref{eq:varpi-sup} equals
  \begin{equation}
    \label{eq:int}
    \mu_{u_x}\bigl([\langle x,u_x\rangle, \langle x,u_x\rangle+r]\bigr)
    = \int_0^r T_{x+tu_x,u_x}\mu\,\mathrm dt.
  \end{equation}
  By the uniform continuity of $T_{x,u}\mu$, 
  \begin{equation}
    \label{eq:varpi-limsup}
    \sup_{x\in\sG(\alpha)}
    \Bigl(\frac{\varpi_+(x,r)}r-T_{x,u_x}\mu\Bigr)\le
    \sup_{x\in\sG(\alpha)}\frac 1r\int_0^r
    \bigl(T_{x+tu_x,u_x}\mu-T_{x,u_x}\mu\bigr)\,\mathrm dt\to0
    \quad \text{as}\; r\downarrow0.
  \end{equation}    
  On the other hand, 
  \begin{align*}
    \varpi_+(x,r)
    &=\sup_{y\in B_r(x)}\bigl(\mu(y+\Hx[y])-\mu(x+\Hx)\bigr)\\
    &\ge\sup_{y\in B_r(x)}\bigl(\mu(y+\Hx)-\mu(x+\Hx)\bigr)
    +\inf_{y\in B_r(x)}\bigl(\mu(y+\Hx[y])-\mu(y+\Hx)\bigr).
  \end{align*}
  Define $W_x(y)=\mu(y+\Hx)$ and 
  \begin{equation}
    \label{eq:W,Delta}
    \Delta_x(y)=\mu(y+\Hx[y])-\mu(y+\Hx)
    =D(y)-W_x(y).
  \end{equation}
  Taking into account the last equality in \eqref{eq:varpi-sup} and the
  representation of its right-hand side in the form of \eqref{eq:int},
  we have
  \begin{multline}
    \label{eq:varpi-liminf}
    \inf_{x\in\sG(\alpha)}
    \Bigl(\frac{\varpi_+(x,r)}r-T_{x,u_x}\mu\Bigr)\ge
    \inf_{x\in\sG(\alpha)}\frac 1r\int_0^r
    \bigl(T_{x+tu_x,u_x}\mu-T_{x,u_x}\mu\bigr)\,\mathrm dt\\
    +\inf_{x\in\sG(\alpha)}\frac 1r
    \inf_{y\in B_r(x)}\Delta_x(y),
  \end{multline}
  where the first summand on the right-hand side vanishes as
  $r\downarrow 0$ for the same reasons as the supremum in
  \eqref{eq:varpi-limsup}.
  
  By (C3), $D$ is differentiable. Moreover, by (C4),
  \begin{equation}
    \label{eq:gradW}
    \grad W_x(y)=\bigl(T_{y,u_x}\mu\bigr)\; u_x, 
  \end{equation}
  which can be checked straightforwardly by considering partial
  derivatives in an orthogonal basis containing $u_x$. So, by
  \eqref{eq:W,Delta}, $\Delta_x(y)$ is differentiable in $y$.
  
  Take some $y^\ast_{x,r}\in\arg\min_{y\in B_r(x)}\Delta_x(y)$. Since
  $\Delta_x(x)=0$, 
  \begin{equation*}
    \bigl|\inf{\!}_{y\in B_r(x)}\Delta_x(y)\bigr|
    =\bigl|\Delta_x(y^\ast_{x,r})-\Delta_x(x)\bigr|=
    \bigl|\langle\grad\Delta_x(z_{x,r}),y^\ast_{x,r}-x\rangle\bigr|
  \end{equation*}
  for some $z_{x,r}$, lying on the segment joining $x$ and
  $y^\ast_{x,r}$, due to the mean value theorem. Hence,
  \begin{equation}
    \label{eq:inf}
    \frac 1r\,\bigl|\inf{\!}_{y\in B_r(x)}\Delta_x(y)\bigr|\le\frac{\|y^\ast_{x,r}-x\|}r\;
    \bigl\|\grad\Delta_x(z_{x,r})\bigr\|\le\bigl\|\grad\Delta_x(z_{x,r})\bigr\|.
  \end{equation}
  Note that $\Delta_x(y)\le0=\Delta_x(x)$ by minimality of
  $\Hx[y]$ at $y$. Hence, $\Delta_x(y)$ has a local maximum at $y=x$,
  and so, $\grad\Delta_x(x)=0$. Thus, by \eqref{eq:inf} and \eqref{eq:W,Delta},
  \begin{align*}
    \frac 1r\,\bigl|\inf{\!}_{y\in B_r(x)}\Delta_x(y)\bigr|&\le
    \bigl\|\grad\Delta_x(z_{x,r})-
    \grad\Delta_x(x)\bigr\|\\&\le
    \bigl\|\grad D(z_{x,r})-
    \grad D(x)\bigr\|+\bigl\|\grad W_x(z_{x,r})-\grad W_x(x)\bigr\|.
  \end{align*}
  Therefore, by \eqref{eq:gradW},
  \begin{align*}
    \bigl|\!\inf_{x\in\sG(\alpha)}\frac 1r\inf{\!}_{y\in
    B_r(x)}\Delta_x(y)\bigr|
    &=\sup_{x\in\sG(\alpha)}\frac 1r\,\bigl|\inf{\!}_{y\in B_r(x)}\Delta_x(y)\bigr|\\
    &\le\sup_{x\in\sG(\alpha)}\bigl\|\grad D(z_{x,r})-
    \grad D(x)\bigr\|+
    \sup_{x\in\sG(\alpha)}\bigl|T_{z_{x,r},u_x}\mu-T_{x,u_x}\mu\bigr|,
  \end{align*}
  which vanishes as $r\downarrow 0$ due to the uniform continuity of
  $\grad D(x)$ and $T_{x,u}\mu$ in $\sG(\alpha)$ and
  $\sG(\alpha)\times\Sphere$, respectively.
  
  This proves that the second summand on the right-hand side of
  \eqref{eq:varpi-liminf} vanishes as $r\downarrow0$.  
  The first claim now follows from \eqref{eq:varpi-limsup} and
  \eqref{eq:varpi-liminf}. The second claim is a consequence of
  \eqref{eq:W,Delta}, \eqref{eq:gradW} and $\grad\Delta_x(x)=0$. 
\end{proof}

Applying Theorem 3.1 from \cite{mol98} in our setting, we arrive at
the following result.

\begin{theorem}
  \label{th:mol}
  Under (C3), (C4), and for $\alpha\in(0,\alpha_{\max})$, the functional $\Phi$
  is continuously differentiable and
  \begin{equation}
    \label{eq:Phi'=}
    \Phi'(\phi)=\sup_{x\in\partial\sR(\alpha)}\frac{|\phi(x)|}{T_{x,u_x}\mu},
  \end{equation}  
  where the outer normal $u_x$ to the minimal halfspace at $x$ is
  defined by \eqref{eq:outnorm}.
\end{theorem}

We will also need the following property of the Hausdorff distance.

\begin{lemma}
  \label{lemma:two-phi}
  Let $\phi=\phi_+-\phi_-$ be the decomposition of $\phi$ into its
  positive and negative parts. Then,
  \begin{equation*}
    \rho_H(\sR(\alpha;\phi),\sR(\alpha))
    = \max\Bigl\{\rho_H\bigl(\sR(\alpha;\phi_+),\sR(\alpha)\bigr),
    \rho_H\bigl(\sR(\alpha;\phi_-),\sR(\alpha)\bigr)\Bigr\}.
  \end{equation*}
\end{lemma}
\begin{proof}
  Let $\rho(x,Y)=\inf_{y\in Y}\rho(x,y)$ stand for the Euclidean
  distance between the point $x$ and the set $Y$.  By the definition
  of the Hausdorff distance, 
  \begin{align*}
    \rho_H\bigl(\sR(\alpha;\phi),\sR(\alpha)\bigr)
    &=\max\Bigl\{\sup_{x\in\sR(\alpha)}\rho(x,\sR(\alpha;\phi_+)),
      \!\!\sup_{x\in\sR(\alpha;\phi_-)}\!\!\rho(x,\sR(\alpha))\Bigr\}\\
    &=\max\Big\{\rho_H(\sR(\alpha;\phi_+),\sR(\alpha)),
    \rho_H(\sR(\alpha;\phi_-),\sR(\alpha))\Big\}. \qedhere
  \end{align*}
\end{proof}

Applying Theorem \ref{th:mol} and Lemma \ref{lemma:two-phi} and
taking into account Remark~\ref{rem:min}, we have the following
result.

\begin{corollary}
  \label{cor:H-r-alpha}
  Under (C3), (C4), and for $\alpha\in(0,\alpha_{\max})$,
  \begin{displaymath}
    \lim_{t\to0}\frac{1}{t}\rho_H\bigl(\sR(\alpha+t),\sR(\alpha)\bigr)
    =\frac{1}{\min_{x\in\partial\sR(\alpha)}T_{x,u_x}\mu}.
  \end{displaymath}
\end{corollary}

\begin{remark}
  \label{rem:K_0}
  Theorem 3.1 in \cite{mol98} dealt with the localised Hausdorff
  distance given by
  $\rho_H\bigl(\sR(\alpha)\cap K_0,\sR(\alpha;\phi)\cap K_0\bigr)$ for
  a fixed compact set $K_0$. In a smooth setting, we do not need
  such a localisation because of compactness of
  $\partial\sR(\alpha)$. However, if (C3) and (C4) are not satisfied
  globally over a neighbourhood of $\partial\sR(\alpha)$ but hold
  locally in a neighbourhood of $\partial\sR(\alpha)\cap K_0$, and,
  for a suitable class of $\phi_n$, the supremum in the definition of
  $\rho_H\bigl(\sR(\alpha),\sR(\alpha;\delta_n\phi_n)\bigr)$ is
  attained inside $K_0$, then $\sup_{x\in\partial\sR(\alpha)}$ in
  \eqref{eq:Phi'=} can be replaced by the supremum over
  $x\in\partial\sR(\alpha)\cap K_0$. This is a typical case, e.g., for
  uniform distributions in polytopes, see Example~\ref{subsec:square}.
\end{remark}

\section{Strong limit theorems in the Hausdorff distance}
\label{sec:strong-limit-theor-Hausdorff}

Since $\emp(\alpha)$ is a random closed set, the Hausdorff distance
$\rho_H(\sR(\alpha),\emp(\alpha))$ is a random variable, see
Theorem~1.3.25(vi) in \cite{mol17}. The next result follows from
Theorem~\ref{th:MZSLLN} for $p=1$ and continuity of $\sR(\alpha)$ in
the Hausdorff metric implied by (C1) and (C2).

\begin{corollary}
  \label{cor:SLLN}
  Under (C1) and (C2), for each $\alpha\in(0,\alpha_{\max})$,
  \begin{equation}
    \label{eq:SLLN}
    \rho_H\bigl(\sR(\alpha),\emp(\alpha)\bigr)\to0\quad \text{
      a.s. as}\; n\to\infty. 
  \end{equation}
\end{corollary}

Note that, for $\xi_1=1$ a.s., a uniform variant of \eqref{eq:SLLN}
(for $\alpha$ from a compact interval in $(0,\alpha_{\max})$) is given
in \cite{dyc17}, Theorem~4.5 and Example~4.2.

\begin{corollary}
  Assume that (C3), (C4) hold for $\alpha\in(0,\alpha_{\max})$, and
  $\E|\xi_1|^p<\infty$ for $p\in[1,2)$.  Then, 
  \begin{displaymath}
    n^{\frac{p-1}p}\rho_H\bigl(\sR(\alpha),\emp(\alpha)\bigr)\to
    0\quad \text{ a.s. as}\; n\to\infty.
  \end{displaymath}
\end{corollary}
\begin{proof}
  Fix any $\eps>0$. By Lemma~\ref{lemma:two-phi} and
  Corollary~\ref{cor:H-r-alpha},
  \begin{displaymath}
    \lim_{n\to\infty}n^{\frac{p-1}p}\rho_H
    \Bigl(\sR\Bigl(\alpha-n^{\frac{1-p}p}\eps\Bigr),
    \sR\Bigl(\alpha+n^{\frac{1-p}p}\eps\Bigr)\Bigr)
    =\frac{\eps}{\min_{x\in\partial\sR(\alpha)}T_{x,u_x}\mu}.
  \end{displaymath}
  Hence, by Theorem \ref{th:MZSLLN},
  \begin{displaymath}
    \limsup_{n\to\infty}n^{\frac{p-1}p}\rho_H
    \bigl(\sR(\alpha),\emp(\alpha)\bigr)\le\frac{\eps}
    {\min_{x\in\partial\sR(\alpha)}T_{x,u_x}\mu}.
  \end{displaymath}
  The result follows by letting $\eps\downarrow0$.
\end{proof}

The next theorem establishes the law of the iterated logarithm for
$\emp(\alpha)$ in terms of the Hausdorff distance. We maintain
notation and assumptions from Section~\ref{sec:LIL}. 

\begin{theorem}
  \label{th:LILH}
  Assume that (C3), (C4) hold for $\alpha\in(0,\alpha_{\max})$. Then
  the following statements hold almost surely:
  \begin{gather}
    \label{eq:lil}
    \limsup_{n\to\infty}\lambda_n^{-1}\rho_H\bigl(\sR(\alpha),\emp(\alpha)\bigr)
    =\frac{\sqrt{M\alpha-\alpha^2}}
    {\min_{x\in\partial\sR(\alpha)}T_{x,u_x}\mu},\\
    \liminf_{n\to\infty}\lambda_n^{-1}
    \rho_H\bigl(\sR(\alpha),\emp(\alpha)\bigr)=0. \notag
  \end{gather}
\end{theorem}
\begin{proof}
  By Theorem~\ref{th:LIL}(i) and Lemma~\ref{lemma:two-phi},
  \begin{equation}
    \label{eq:Hbound}
    \rho_H\bigl(\sR(\alpha),\emp(\alpha)\bigr)
    \leq \max\bigl\{\rho_H\bigl(\sR(\alpha),\sR(\alpha+\gamma\lambda_n)\bigr),
      \rho_H\bigl(\sR(\alpha),\sR(\alpha-\gamma\lambda_n)\bigr)\bigr\}
  \end{equation}
  for any $\gamma>\sqrt{M\alpha-\alpha^2}$ and all sufficiently large
  $n$. By Corollary~\ref{cor:H-r-alpha}, 
  \begin{displaymath}
    \limsup_{n\to\infty}\lambda_n^{-1}
    \rho_H\bigl(\sR(\alpha),\emp(\alpha)\bigr)\le
    \frac{\gamma}{\min_{x\in\partial\sR(\alpha)}T_{x,u_x}\mu}\quad \text{ a.s.}
  \end{displaymath}
  Letting $\gamma\downarrow\sqrt{M\alpha-\alpha^2}$ yields the upper
  bound in the first claim.

  To confirm the lower bound, fix some $\gamma<\sqrt{M\alpha-\alpha^2}$
  and let $x\in\partial\sR(\alpha)$ be an arbitrary point from
  $\arg\min_{x\in\partial\sR(\alpha)}T_{x,u_x}\mu$. By (C3), there is a
  unique minimal halfspace $H$ at $x$ with the outer normal $u_x$.
  As in the proof of the
  right-hand non-inclusion in (ii) of Theorem~\ref{th:LIL}, we may
  choose a sequence $n''_k\to\infty$ and arbitrary points
  $x_k\in\partial\sR(\alpha-\gamma\lambda_{n''_k})$ such that
  $x_k\to x$ and
  $x_k\in\partial\sR(\alpha-\gamma\lambda_{n''_k})\cap\emp[n''_k](\alpha)$
  for large $k$. However, this time we choose the points $x_k$ in a more
  specific way, namely, as intersections of $\ell=\{x+tu_x, t>0\}$ with
  $\partial\sR(\alpha-\gamma\lambda_{n''_k})$.  Slightly moving $x_k$
  along $\ell$ as in the aforementioned proof, we obtain a sequence
  of points $\tilde x_k\in\ell$ such that
  \begin{displaymath}
    \tilde
    x_k\in\emp[n''_k](\alpha)\setminus\sR(\alpha-\gamma\lambda_{n''_k})
  \end{displaymath}
  for all sufficiently large $k$.

  By the definition of the Hausdorff metric and due to
  the fact that $\ell$ is orthogonal to $\partial H$,
  \begin{equation}
    \label{eq:rhoHrho}
    \rho_H\bigl(\emp[n''_k](\alpha),\sR(\alpha)\bigr)\ge
    \rho\bigl(\tilde x_k,\sR(\alpha)\bigr)\ge \rho(x_k,x).
  \end{equation}
  Since $D(x_k)-D(x)=-\gamma\lambda_{n''_k}$, we have
  $\varpi_-\bigl(x,\rho(x_k,x)\bigr) \le-\gamma\lambda_{n''_k}$, where
  $\varpi_-$ is defined in \eqref{eq:varpi_-}.  As $\varpi_-(x,\cdot)$
  strictly decreases for small $r$, there exists
  $\varpi_-^{\leftarrow}(x,\cdot)$, an inverse function in $r$, which also
  decreases in $r$. Hence,
  $\rho(x_k,x)\ge\varpi_-^{\leftarrow}(x,-\gamma\lambda_{n''_k})$, and,
  by \eqref{eq:rhoHrho},
  \begin{equation}
    \label{eq:lowbound}
    \begin{aligned}
      \limsup_{n\to\infty}\lambda_n^{-1}\rho_H\bigl(\sR(\alpha),\emp(\alpha)\bigr)
      &\ge\limsup_{k\to\infty}\lambda_{n''_k}^{-1}\,
      \rho_H\bigl(\sR(\alpha),\emp[n''_k](\alpha)\bigr)\\
      &\ge\limsup_{k\to\infty}\lambda_{n''_k}^{-1}\,
      \varpi_-^{\leftarrow}(x,-\gamma\lambda_{n''_k}).
    \end{aligned}
  \end{equation}
  It easily follows from Lemma \ref{lem:varpi} that
  \begin{displaymath}
    \lim_{t\downarrow0}t^{-1}\varpi_-^{\leftarrow}(x,-\gamma t)
    =\frac\gamma{T_{x,u_x}\mu}.
  \end{displaymath}
  Thus, the right-hand side in \eqref{eq:lowbound} equals
  $\gamma/T_{x,u_x}\mu$. In view of
  $x\in\arg\min_{x\in\partial\sR(\alpha)}T_{x,u_x}\mu$, letting
  $\gamma$ grow to $\sqrt{M\alpha-\alpha^2}$ yields the lower bound in
  the first claim.

  The second claim follows from (iii) in Theorem~\ref{th:LIL} in the
  same way as the upper bound of the first claim followed from (i).
\end{proof}

\section{Examples}
\label{sec:ex}

In this section, we provide some examples of applying the law of the
iterated logarithm to specific distributions.

\begin{example}[Univariate distribution]
  On the line, $D(x)$ is the minimum of $\mu((-\infty,x])$ and
  $\mu([x,\infty))$, and the depth trimmed region of a probability
  measure $\mu$ with connected support is the interval
  $[q_\alpha(\mu),q_{1-\alpha}(\mu)]$ between its quantiles.  Then, in
  the non-weighted case, our results follow  from well-known strong limit
  theorems for empirical quantiles, see, e.g.,
  \cite[Theorem~5]{MR0501290}.
\end{example}

\begin{example}[Uniform distribution]
  Assume that $\mu$ is the uniform distribution on a convex body
  $K$. If $K$ is symmetric smooth and strictly convex, then
  Theorem~\ref{thr:uniform} yields that the depth function is
  continuously differentiable. Hence, all strong limit theorems in the
  inclusion and metric sense hold. For instance, Theorem~\ref{th:LILH}
  is applicable with the right-hand side of \eqref{eq:lil} given by
  \begin{displaymath}
    \frac{\sqrt{M\alpha-\alpha^2}\,V_d(K)}
    {\min_{x\in\partial\sR(\alpha)}V_{d-1}(K\cap(x+\partial H_{u_x}))}.
  \end{displaymath}
  For instance, let $\mu$ be the uniform distribution in the unit disk $K$
  centred at the origin.  It is shown in Section~5.6 in
  \cite{rous:rut99} that $\sR(\alpha)=\bigl\{x\in\R^2\colon\|x\|\le
  r(\alpha)\bigr\}$, where $\alpha\le1/2$
  and $r(\alpha)$ is the unique root of the equation
  \begin{displaymath}
    \arcsin r+r\sqrt{1-r^2}=\frac\pi2-\pi\alpha.
  \end{displaymath}
  Hence, for $\alpha\in(0,1/2)$ and $x\in\partial\sR(\alpha)$, we have
  $V_2(K)=\pi$ and
  $V_1(K\cap(x+\partial H_{u_x}))=2\sqrt{1-r^2(\alpha)}$.  Thus,
  Theorem~\ref{th:LILH} holds for $\alpha\in(0,1/2)$ with the
  right-hand side of \eqref{eq:lil} being
  \begin{displaymath}
    \frac\pi2\sqrt{\frac{M\alpha-\alpha^2}{1-r^2(\alpha)}}.
  \end{displaymath}
\end{example}

\begin{example}[Uniform distribution on a square]
  \label{subsec:square}
  Let $\mu$ be the uniform distribution on the unit square
  $K$. Theorem~\ref{thr:uniform} is not applicable since $K$ is
  neither smooth nor strictly convex.  Due to Section~5.4 in
  \cite{rous:rut99},
  \begin{gather*}
    D(x)=2\min\{x_1,1-x_1\}\min\{x_2,1-x_2\},\qquad	x=(x_1,x_2)\in K,\\
    \sR(\alpha)=\bigl\{x\in\R^2\colon
    \min\{x_1,1-x_1\}\min\{x_2,1-x_2\}\ge\alpha/2\bigr\},\qquad\alpha\le1/2.
  \end{gather*}
  The depth function $D$ is not globally differentiable, and
  $\partial\sR(\alpha)$ is not smooth, see
  Figure~\ref{fig:Unif_depth}.  Thus, we cannot apply
  Theorem~\ref{th:LILH} in its original form.

  \begin{figure}[htbp]
  \centering
  \vspace{1pt}
  \includegraphics[scale=0.5]{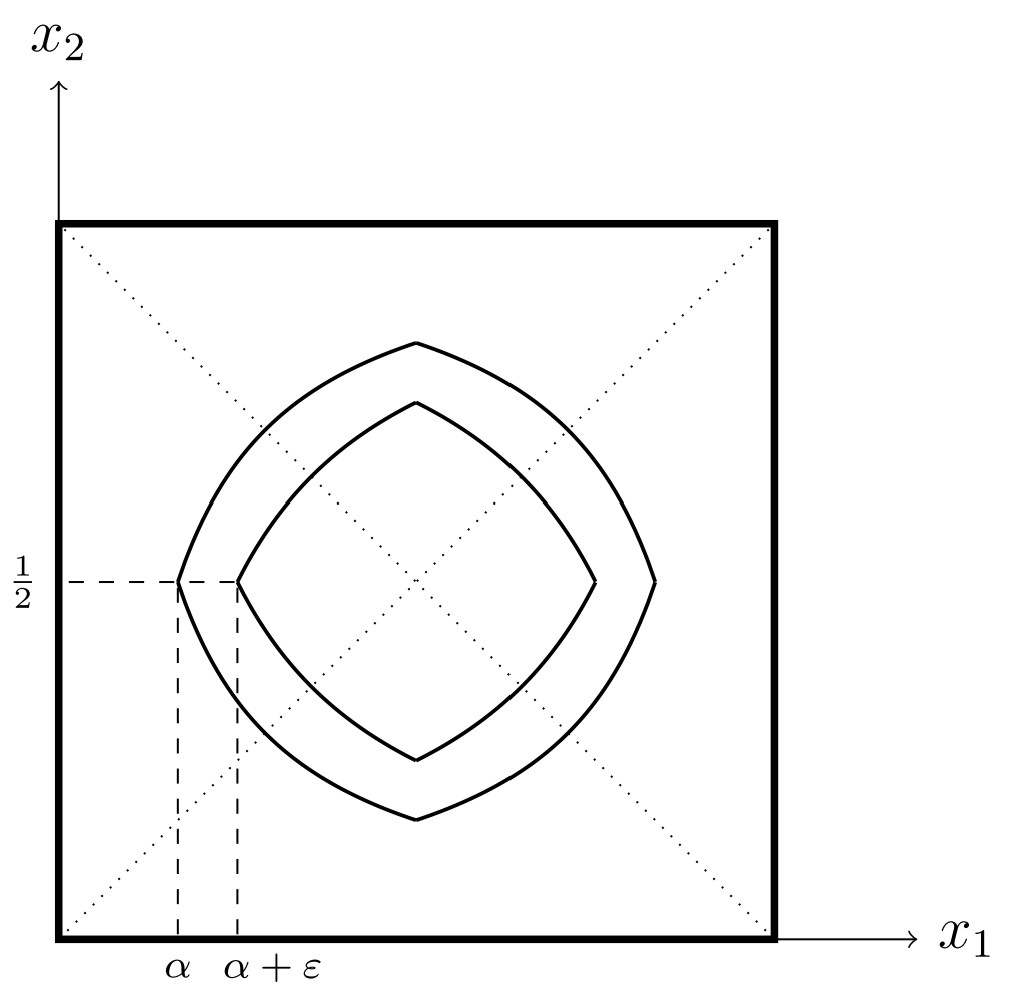}
%    \begin{tikzpicture}[scale=5]
%      \draw[->] (0,0) -- (1.2,0) node[right] {$x_1$};
%      \draw[->] (0,0) -- (0,1.2) node[above] {$x_2$};
%      \draw[line width=1.5pt] (0,0) rectangle (1,1);
%      \draw[line width=.75pt, domain=1/6:1/2, smooth, variable=\x] plot ({\x}, {1/(12*\x)});
%      \draw[line width=.75pt, domain=1/6:1/2, smooth, variable=\x] plot ({\x}, {1-1/(12*\x)});
%      \draw[line width=.75pt, domain=1/2:5/6, smooth, variable=\x] plot ({\x}, {1-1/(12*(1-\x))});
%      \draw[line width=.75pt, domain=1/2:5/6, smooth, variable=\x] plot ({\x}, {1/(12*(1-\x))});
%      \draw[line width=.75pt, domain=1/4:1/2, smooth, variable=\x] plot ({\x}, {1/(8*\x)});
%      \draw[line width=.75pt, domain=1/4:1/2, smooth, variable=\x] plot ({\x}, {1-1/(8*\x)});
%      \draw[line width=.75pt, domain=1/2:3/4, smooth, variable=\x] plot ({\x}, {1-1/(8*(1-\x))});
%      \draw[line width=.75pt, domain=1/2:3/4, smooth, variable=\x] plot ({\x}, {1/(8*(1-\x))});
%      \draw[dashed, dash phase=5pt] (1/6,0) node[below] {\tiny $\alpha$} -- (1/6,1/2);
%      \draw[dashed, dash phase=5pt] (1/4,0) node[below right, xshift=-7pt, yshift=1.3pt] {\tiny $\alpha+\eps$} -- (1/4,1/2);
%      \draw[dashed, dash phase=4pt] (0,1/2) node[left] {\tiny $\frac12$} -- (1/4,1/2);
%      \draw[dotted] (0,0) -- (1,1);
%      \draw[dotted] (1,0) -- (0,1);
%    \end{tikzpicture}
  \caption{Boundaries of $\sR(\alpha)$ and $\sR(\alpha+\eps)$ for the
    uniform distribution in the unit square. }
  \label{fig:Unif_depth}
  \end{figure}

  It can be easily checked that the supremum in the definition of
  $\rho_H\bigl(\sR(\alpha),\sR(\alpha+\eps)\bigr)$ is
  attained on the diagonals $x_1=x_2$ and $x_1+x_2=1$. Note that, in a small
  neighbourhood $K_0$ of
  $\bigl(\sqrt{\frac\alpha2},\sqrt{\frac\alpha2}\bigr)$, which is the
  intersection point of $\partial\sR(\alpha)$ and $x_1=x_2$,
  conditions (C3) and (C4) hold. Hence, Remark ~\ref{rem:K_0} applies with
  \begin{displaymath}
    \min_{x\in\partial\sR(\alpha)\cap K_0}T_{x,u_x}\mu=2\sqrt\alpha,
  \end{displaymath}
  being the length of the segment tangent to $\partial\sR(\alpha)$ at
  the point $\bigl(\sqrt{\frac\alpha2},\sqrt{\frac\alpha2}\bigr)$.
  Therefore, Theorem~\ref{th:LILH} holds for $\alpha\in(0,1/2)$ with
  $\limsup$ given by $\frac12\sqrt{M-\alpha}$. The same approach can
  be applied for any uniform distribution $\mu$ on a polygon, since,
  in this case, $\partial\sR(\alpha)$ consists of a finite number of
  hyperbolic arcs, see Sections~5.3-5.5 in \cite{rous:rut99}.
\end{example}

\begin{example}[Normal distribution]
  Let $\mu$ be the standard bivariate normal distribution in $\R^2$.
  Due to spherical symmetry of $\mu$, we have
  $D(x)=\frac 12-\Phi_0\bigl(\|x\|\bigr)$, $x\in\R^2$,
  where
  \begin{displaymath}
    \Phi_0(t)=\frac 1{\sqrt{2\pi}}\int_0^t\mathrm e^{-s^2/2}\,\mathrm ds.
  \end{displaymath}
  Hence,
  \begin{displaymath}
    \sR(\alpha)=\bigl\{x\in\R^2\colon\|x\|\le\Phi_0^\leftarrow(1/2-\alpha)\bigr\},
    \quad\alpha\le1/2,
  \end{displaymath}
  is the centred disk with radius given by the inverse of $\Phi_0$,
  see also Section~5.7 in \cite{rous:rut99}.  For
  $\alpha\in(0,1/2)$ and $x\in\partial\sR(\alpha)$, we have
  \begin{displaymath}
    T_{x,u_x}\mu=\biggl.\frac 1{2\pi}\int_{-\infty}^\infty
    \mathrm e^{-\frac{1}{2}(s^2+t^2)}\,\mathrm dt\,\biggr|_{s=\Phi_0^\leftarrow(1/2-\alpha)}=
    \frac 1{\sqrt{2\pi}}\mathrm
    e^{-\frac{1}{2}(\Phi_0^\leftarrow(1/2-\alpha))^2}.
  \end{displaymath}
  Therefore, Theorem~\ref{th:LILH} holds for $\alpha\in(0,1/2)$ with
  the right-hand side of \eqref{eq:lil} being
  \begin{displaymath}
    \sqrt{2\pi(M\alpha-\alpha^2)}\;\mathrm
    e^{\frac{1}{2}(\Phi_0^\leftarrow(1/2-\alpha))^2}.
  \end{displaymath}
\end{example}

\begin{example}[Cauchy distribution]
  Let $\mu$ be the bivariate Cauchy distribution with the density
  \begin{displaymath}
    f(x)=\frac 1{\pi^2}\frac 1{(1+x_1^2)(1+x_2^2)},\qquad
    x=(x_1,x_2)\in\R^2.
  \end{displaymath}
  Due to Section 5.8 in \cite{rous:rut99},
  \begin{gather*}
    D(x)=\frac 12-\frac 1\pi\arctan\max\bigl\{|x_1|,|x_2|\bigr\},\qquad x\in\R^2,\\
    \sR(\alpha)=\bigl\{x\in\R^2\colon
    \max\{|x_1|,|x_2|\}\le\cot\pi\alpha\bigr\},\qquad\alpha\le1/2,
  \end{gather*}
  which is the square with vertices
  $(\pm\cot\pi\alpha,\pm\cot\pi\alpha)$. The Hausdorff distance
  $\rho_H\bigl(\sR(\alpha),\sR(\alpha+\eps)\bigr)$
  equals the distance between the vertices of the corresponding
  squares, where the depth function $D$ is not
  differentiable. Hence, Theorem~\ref{th:LILH} is not applicable even
  with Remark~\ref{rem:K_0} taken into account.  However, these
  vertices are located on fixed diagonals $x_1\pm x_2=0$,
  and so the Hausdorff distance can be calculated explicitly as
  \begin{equation}
    \label{eq:HCauchy}
    \rho_H\bigl(\sR(\alpha),\sR(\alpha+\eps)\bigr)=
    \sqrt 2\,\bigl|\cot\pi(\alpha+\eps)-\cot\pi\alpha\bigr|.
  \end{equation}
  Then \eqref{eq:Hbound} and \eqref{eq:HCauchy} yield
  \begin{multline*}
    \rho_H\bigl(\sR(\alpha),\emp(\alpha)\bigr)
    \le\sqrt 2\max\bigl\{\cot\pi(\alpha-\gamma\lambda_n)-\cot\pi\alpha,
      \cot\pi\alpha-\cot\pi(\alpha+\gamma\lambda_n)\bigr\}\\
      =\frac{\pi\sqrt 2\,\gamma\lambda_n}{\sin^2\pi\alpha}+o(\lambda_n)
      \quad \text{as}\; n\to\infty,
  \end{multline*}
  for any $\gamma>\sqrt{M\alpha-\alpha^2}$. Hence,
  \begin{equation*}
    \limsup_{n\to\infty}\lambda_n^{-1}\rho_H\bigl(\sR(\alpha),\emp(\alpha)\bigr)
    \le\frac{\pi\sqrt{2(M\alpha-\alpha^2)}}{\sin^2\pi\alpha}.
  \end{equation*}

  To prove the reverse inequality, we can not directly follow the
  approach used in the proof of Theorem \ref{th:LILH} since a minimal
  halfspace at a vertex of $\partial\sR(\alpha)$ is not unique. But,
  with some loss of optimality, we can apply the same technique to any
  point $x$ of the boundary which is not a vertex.  In this case,
  \[T_{x,u_x}\mu=\biggl.\frac 1{\pi^2}\int_{-\infty}^\infty
    \frac 1{(1+s^2)(1+t^2)}\,\mathrm dt\,\biggr|_{s=\cot\pi\alpha}=\frac{\sin^2\pi\alpha}{\pi},\]
  which results in a lower bound.
%  \[\frac{\pi\sqrt{M\alpha-\alpha^2}}{\sin^2\pi\alpha}.\]
  Therefore, Theorem \ref{th:LILH} holds for $\alpha\in(0,1/2)$ with
  \eqref{eq:lil} replaced by the inequality
  \begin{displaymath}
    \frac{\pi\sqrt{M\alpha-\alpha^2}}{\sin^2\pi\alpha}\le
    \limsup_{n\to\infty}\lambda_n^{-1}\rho_H\bigl(\sR(\alpha),\emp(\alpha)\bigr)\le
    \frac{\pi\sqrt{2(M\alpha-\alpha^2)}}{\sin^2\pi\alpha}.
  \end{displaymath}
  The same approach applies to further symmetric distributions, where
  $\sR(\alpha)$ is also a square, see Example~2 in
  \cite{nag:sch:wer18}.
\end{example}

\section{Appendix: Depth functions of convex bodies}
\label{sec:appendix}

Consider a convex body $K$ in $\R^d$, and let $\mu$ be the uniform
probability distribution on $K$, that is, $\mu$ is the Lebesgue
measure on $K$ normalised by the volume $V_d(K)$ of $K$. By rescaling,
in all subsequent arguments it is possible to let $V_d(K)=1$. The
depth trimmed region of $\mu$ at level $\alpha$ is said to be the
\emph{convex floating body} $K_\alpha$ of $K$, see \cite{schut:wer90}.
Then
\begin{displaymath}
  D(x)=\sup\{\alpha\colon x\in K_\alpha\},\quad x\in K,
\end{displaymath}
is the depth function of $\mu$.

For a convex set $L$ in $\R^d$ denote by $C(L)$ its centroid, that is,
the integral $\int_L x\,\mathrm dx$ (taken with respect to the Lebesgue
measure on the affine subspace generated by $L$) normalised by the
Lebesgue measure of $L$ taken in its affine hull.

For a set $L$ and $r>0$ denote by
\begin{displaymath}
  L^r=\{x\in\R^d\colon\rho(x,L)\leq r\}\quad\text{ and }\quad
  L^{-r}=\{x\in\R^d\colon B_r(x)\subset L\}
\end{displaymath}
the outer and inner $r$-envelopes of $L$. If $L$ is a convex set, then
the Steiner formula says
\begin{displaymath}
  V_d(L^r)=\sum_{i=0}^d r^{d-i} \kappa_{d-i} V_i(L),
\end{displaymath}
where the coefficients $V_i(L)$, $i=0,\dots,d$, are said to be
intrinsic volumes of $L$ and $\kappa_j$ is the volume of the unit
Euclidean ball in $\R^j$, see \cite[Section~4.1]{schn2}. Furthermore,
denote $\|L\|=\sup\{\|x\|\colon x\in L\}$ and
\begin{displaymath}
  m(L)=\sum_{i=0}^d \kappa_{d-i} V_i(L). 
\end{displaymath}
If $E$ is a linear subspace of $\R^d$, denote by $\proj_E x$ the
orthogonal projection of $x$ onto $E$. For each $w,v\in\Sphere$,
denote by $w_{v}$ the normalised projection of $w$ on
$\partial H_{v}$.

\begin{lemma}
  \label{lemma:1}
  Assume that $K$ is a convex body in $\R^d$, and let $\delta>0$ be
  such that $K^{-\delta}$ is not empty. For each pair of 
  $(d-1)$-dimensional affine hyperplanes $E_1$ and $E_2$ which
  intersect $K^{-\delta}$, we have that
  \begin{displaymath}
    \rho_H(K\cap E_1, K\cap E_2)\leq c_1 \rho_H(E_1,E_2)
  \end{displaymath}
  and
  \begin{displaymath}
    \|C(K\cap E_1)-C(K\cap E_2)\|\leq c_2 \rho_H(E_1,E_2)
  \end{displaymath}
  with constants $c_1$ and $c_2$ depending only on $K$ and $\delta$.
\end{lemma}
\begin{proof}
  It suffices to consider only parallel $E_1$ and $E_2$ since
  otherwise the Hausdorff distance between them is infinite. Denote
  $\eps=\rho_H(E_1,E_2)$. 
  Let $u$ be the unit normal to $E_1$ such that
  $E_2=E_1+\eps u$. Furthermore, let $x$ be any support point of
  $K$ in direction $u$, that is, $\langle x,u\rangle$ equals the
  support function
  \begin{displaymath}
    h(K,u)=\sup\{\langle y,u\rangle: y\in K\},
  \end{displaymath}
  and let $M$ be the convex hull of $x$ and $K\cap E_1$. Since $E_1$
  intersects $K^{-\delta}$, the distance between $x$ and $E_1$ is at
  least $\delta$, so that $M$ is a cone of height at least
  $\delta$. Since $K\cap E_1=M\cap E_1$ is obtained by scaling $M\cap
  E_2$ with a factor at most $(1+\eps/\delta)$ and the displacement of any
  point from $K$ is at most $2\|K\|\eps/\delta$,
  \begin{displaymath}
    K\cap E_1\subset(M\cap E_2)^{2\delta^{-1}\|K\|\eps}\subset
   (K\cap E_2)^{2\delta^{-1}\|K\|\eps}.
  \end{displaymath}
  The same argument with the roles of $E_1$ and $E_2$ interchanged and
  with $x$ being the support point in direction $(-u)$ yields the
  result for the Hausdorff distance with $c_1=2\delta^{-1}\|K\|$.

  By assumption, we have that
  $V_{d-1}(K\cap E_i)\geq \delta^{d-1}\kappa_{d-1}$, $i=1,2$.  For all
  $r\leq 1$, 
  \begin{displaymath}
    V_{d-1}((K\cap E_i)^r)-V_{d-1}(K\cap E_i)\leq rm(K\cap E_i)
    \leq rm(K),
    \quad i=1,2,
  \end{displaymath}
  where we used the monotonicity of intrinsic volumes, see
  \cite[Eqn.~(5.25) and (5.31)]{schn2}.
  Without loss of generality, we may assume that $c_1\eps\le1$. Then,
  due to the first part of the lemma,
  \begin{displaymath}
    |V_{d-1}(K\cap E_1)-V_{d-1}(K\cap E_2)|
    \leq c_1\eps m(K).
  \end{displaymath}
  Denote $L=(K\cap E_2)-\eps u$, where the translation ensures that
  $L\subset E_1$. Thus,
  \begin{align*}
    \|C &(K\cap E_1)-C(K\cap E_2)\|\\
    &\leq \frac{1}{V_{d-1}(K\cap E_1)}
      \Big\|\int_{K\cap E_1} x\,\mathrm dx-\int_{L} (x+\eps u)\,\mathrm dx\Big\|\\
    &\qquad\qquad\qquad \qquad\qquad\qquad
    +\frac{|V_{d-1}(K\cap E_1)-V_{d-1}(K\cap E_2)|}
    {V_{d-1}(K\cap E_1) V_{d-1}(K\cap E_2)}
    \Big\|\int_{K\cap E_2}x\,\mathrm dx\Big\|\\
    &\leq \frac{1}{\delta^{d-1}\kappa_{d-1}}\Big(
      \int_{(K\cap E_1)\Delta L}\|x\|\,\mathrm dx\Big)
      +\frac{1}{\delta^{d-1}\kappa_{d-1}}\Big(
      \int_{L} \eps\,\mathrm dx
      +c_1\eps m(K)\|C(K\cap E_2)\|\Big)\\
    &\leq \frac{1}{\delta^{d-1}\kappa_{d-1}}\Big(
     2c_1\eps m(K)\|K\|+\eps m(K)(c_1\|K\|+1)\Big)=c_2\eps, 
  \end{align*}
  where $c_2$ depends only on $K$ and $\delta$.   
\end{proof}

The following result for $K$ with $C^1$-boundary can be recovered from
the proof of Lemma~5 in \cite{MR1270561}, however, without a uniform
bound on the Lipschitz constant. Recall that $u_y$ denotes a unit
normal vector to the convex floating body $K_\alpha$ at the point $y$
on its boundary. 

\begin{lemma}
  \label{lemma:2}
  Let $K$ be a convex body in $\R^d$, and let $\delta>0$ be such that
  $K^{-\delta}$ is not empty. Then
  \begin{displaymath}
    \mu(y+H_v)-\mu(y+H_{u_y})\leq c\|v-u_y\|^2, \quad y\in K^{-\delta},
  \end{displaymath}
  for a constant $c$ which depends only on $\delta$ and $K$. 
\end{lemma}
\begin{proof}
  Applying a translation, assume that $y=0$ and denote $u=u_y$.  Let
  $L=K\cap \partial H_u$.  Furthermore, let $L_1$ be the set of
  $x\in L$ such that $\proj_{\partial H_v} x\in K$, and let $L_2$ be
  the set of all $x\in\partial H_u$ such that the segment which joins
  $x$ and $y=\proj_{\partial H_v} x$ intersects $K$. Note that
  $L_1\subset L\subset L_2$. 
    
  In the following we use the inequality
  \begin{displaymath}
    \|z-\proj_{\partial H_v} z\|
    =|\langle z,v\rangle|
    =|\langle z,v-u\rangle|\leq \|z\|\|v-u\|, \quad z\in \partial
    H_{u}.
  \end{displaymath}
  Denote by $r_M(w)=\sup\{t: tw\in M\}$ the radial function of a
  convex set $M$ in direction $w\in\Sphere$.

  Consider $w\in \Sphere \cap \partial H_u$ and $z=r_{L_2}(w)w$, which
  belongs to the relative boundary of $L_2$. Let $y$ belong to
  the intersection of $K$ and the segment which joins $z$ and
  $z_1=\proj_{\partial H_v} z$. Then $y$ belongs to the boundary of
  $K$ and the line passing through $z$ and $z_1$ is a tangent line to
  $K$, so that $y$ belongs to 
  $\partial H_{v'}$ with $v'$ being a convex combination of $u$ and
  $v$. Then $y=r_K(w_{v'})w_{v'}$.  Since
  $r_K(w)=r_L(w)$ and $\|y-z\|\leq \|z_1-z\|$,
  \begin{align*}
    |r_L(w)-r_{L_2}(w)|
    &\leq \|r_K(w)w-y\|+\|y-z\|\\
    &\leq \|r_K(w)w-r_K(w_{v'})w\|
      + \|r_K(w_{v'})w-r_K(w_{v'})w_{v'}\|
      + \|\proj_{\partial H_{v'}} z -z\|\\
    &\leq |r_K(w)-r_K(w_{v'})|
      +r_K(w_{v'}) \|w-w_{v'}\|
      + \|z\|\|u-v\|.
  \end{align*}
  The radial function $r_K$ is Lipschitz, that is,
  $|r_K(w_{v'})-r_K(w)|\leq c_3\|w_{v'}-w\|$, with a constant $c_3$
  depending only on $\delta$ and $K$. This is seen by writing
  $r_K(u)=h(K^o,u)^{-1}$ (see \cite[Lemma~1.7.13]{schn2}), using the
  support function of the polar body to $K$ and noticing that $K^o$ is
  contained in the ball of radius $\delta^{-1}$ and contains the ball
  of radius $\|K\|^{-1}$ and the fact that the support function is
  Lipschitz. Note also that $r_K(w)$ is bounded by $\|K\|$.
%  a constant which depends on $K$.
  Thus, and since $\|w_{v'}-w\|\leq \|v'-u\|\leq\|v-u\|$, we have
  \begin{displaymath}
    |r_L(w)-r_{L_2}(w)|\leq c_4 \|v-u\|
  \end{displaymath}
  for all $w\in\Sphere \cap \partial H_u$. Since the Hausdorff
  distance between convex sets is bounded by the uniform distance between
  their radial functions, $\rho_H(L,L_2)\leq c_4\|u-v\|$. 
  
  Now consider the set $L_1$, take $w\in \Sphere \cap \partial H_u$,
  and let $z_1=\proj_{\partial H_v} z$ for $z=r_L(w)w$. If $z_1\in K$,
  then $z$ lies on the boundary of $L_1$ and $r_{L_1}(w)=r_L(w)$. If
  $z_1\notin K$, then the point $z_2\in L$ with
  $r_K(w_v)w_v=\proj_{\partial H_v} z_2$ lies on the boundary of $L_1$, so
  that
  \begin{align*}
    r_L(w)-r_{L_1}(w)
    &=\|r_K(w)w-z_2\|\\
    &\leq \|r_K(w)w-r_K(w_v) w_v\|
      +\|r_K(w_v) w_v-z_2\|\\
    &\leq |r_K(w)-r_K(w_v)|+ r_K(w)\|w-w_v\|
      +\|\proj_{\partial H_v} z_2-z_2\|\leq c_5 \|u-v\|.
  \end{align*}
  Arguing as above, we obtain $\rho_H(L,L_1)\leq c_5\|v-u\|$.
  Combining this with the bound on $\rho_H(L,L_2)$ yields that
  \begin{equation}
    \label{eq:L12}
    L_2\subset L_1+c_6\|u-v\| B^{d-1},
  \end{equation}
  where $B^{d-1}$ is the unit ball in the space $\partial H_u$ of
  dimension $(d-1)$.
  
  Denote
  \begin{align*}
    M&=\{z\in \R^d\colon\langle z,u\rangle \langle z,v\rangle\leq 0\},\\
    h&=\sup\{\|z-\proj_{\partial H_v} z\|\colon z\in L_2\}.
  \end{align*}
  Note that $h\leq c_7\|u-v\|$. Furthermore, let
  \begin{math}
    Z=M\cap \conv(L\cup \proj_{\partial H_v}L).
  \end{math}
  Then 
  \begin{align*}
    \mu(H_v)-\mu(H_{u})
    \leq |V_d(Z\cap H_v)-V_d(Z \cap H_u)|
    +V_d(M\cap(K\triangle Z)).
  \end{align*}
  Note that $M\cap(K\triangle Z)$ is a subset of the difference of two
  cylinders, $C_2=L_2+[-hv,hv]$ and $C_1=L_1+[-hv,hv]$. By
  \eqref{eq:L12},
  \begin{align*}
    V_d(M\cap(K\triangle Z))\leq V_d(C_2)-V_d(C_1)
    &= 2h\big(V_{d-1}(L_2)-V_{d-1}(L_1)\big)\\
    &=2h\sum_{i=0}^{d-2} (c_6\|u-v\|)^{d-1-i} \kappa_{d-1-i} V_i(L_1)
    \leq c_8\|u-v\|^2.
  \end{align*}
%  Choose a unit vector $w\in H_v$ such that $w$ is orthogonal to
%  $H_v\cap H_u$.
  Finally,
  \begin{displaymath}
    V_d(Z\cap H_{u})-V_d(Z\cap H_v)
    =\langle u,v\rangle \int_{z\in L} \langle z,v\rangle \mathrm{d}z
    =\langle u,v\rangle \Big\langle \int_{z\in L} z \mathrm{d}z,v\Big\rangle =0,
  \end{displaymath}
  where $\langle u,v\rangle$ is the Jacobian of the linear
  transformation used in the integral, and the last integral vanishes
  since $0$ is the centroid of $L=K\cap H_u$, see, e.g., 
  \cite[Lemma~2]{MR1270561}.
\end{proof}

A convex body is said to be smooth if each point on its boundary
admits a unique outer normal; it is symmetric if $K$ equals its
reflection with respect to some point which is usually put at the
origin; $K$ is strictly convex if its boundary does not contain any
segment.  It is known that if $K$ is symmetric smooth and strictly
convex, then the convex floating body $K_\alpha$ is of the class $C^2$
for all $0<\alpha<V_d(K)/2=\alpha_{\max}$, see
\cite[Theorem~3]{mey:rei91}.  Then, each $x\neq 0$ from the interior
of $K$ belongs to some $\partial K_\alpha$ and there exists a unique
outer normal vector $u_x$ to $K_\alpha$ such that
$\mu(x+H_{u_x})=\alpha$, see Lemma \ref{lem:level}.

\begin{lemma}
  \label{lemma:3}
  Assume that $K$ is a symmetric smooth strictly convex body. Then the
  function $x\mapsto u_x$ is Lipschitz on
  $\{x\in K\colon \delta_1 \leq D(x)\leq \delta_2\}$, where
  $0<\delta_1<\delta_2< V_d(K)/2$, and the Lipschitz constant
  depends on $K$ and $\delta_1,\delta_2$.
\end{lemma}
\begin{proof}
  Assume that $D(x)=t$ and $D(y)=s$, where
  $\delta_1\leq t\leq s\leq \delta_2$ without loss of generality.

  We first settle the case when $t=s$, so that both $x$
  and $y$ belong to the boundary of $K_t$ for some
  $t\in[\delta_1,\delta_2]$. It is known from
  \cite[Eq.~(19)]{MR0919390} and \cite[Lemma~4]{MR1270561} that the
  radius of curvature $R(x)$ of $K_t$ at a point $x$ on its boundary equals
  the determinant of the matrix $Q$ given by
  \begin{displaymath}
    Q_{i,j}=\frac{1}{V_{d-1}(K\cap E)}
    \int_{\mathbb{S}^{d-2}}\eta_i\eta_j r(\eta)^d \cot\beta(\eta)
    \mathrm{d}\eta,\quad i,j=1,\dots,d-1,
  \end{displaymath}
  where $E$ is the $(d-1)$-dimensional affine space passing
  through $x$ with normal $u_x$, that is, the tangent plane to $K_t$
  at $x$, the integration is over
  the unit sphere in $E$, and $r(\eta)$ is the radial function of
  $K\cap E$. The angle $\beta(\eta)$ is built by the
  half-line $\ell(\eta)$ through $x$ in direction $\eta\in E$ and the
  tangent half-line to $\partial K$ at $y=\ell(\eta)\cap\partial K$,
  whose orthogonal projection onto $E$ is $\ell(\eta)$ and which
  belongs to the complement of $x+H_{u_x}$, see Figure~2 in
  \cite{mey:rei91}.

  The determinant of $Q$ is bounded below by the $(d-1)$th power of the
  infimum of
  \begin{displaymath}
    \frac{1}{V_{d-1}(K\cap E)}
    \int_{\mathbb{S}^{d-2}}\langle \eta,y\rangle^2
    r(\eta)^d \cot\beta(\eta) \mathrm{d}\eta
  \end{displaymath}
  for all $y$ from the unit sphere $\mathbb{S}^{d-2}$ in $E$. As in
  the proof of \cite[Theorem~3]{mey:rei91}, let $\tilde{K}$ be the
  convex hull of $K\cap E$ and $K\cap(-E)$, and let
  $\tilde\beta(\eta)$ be the maximum of the angles constructed for
  $\tilde K$ in the same way as $\beta(\eta)$ was constructed for
  $K$. Since $t\leq\delta_2<V_d(K)/2$ and $K$ is strictly convex, we
  have $\tilde\beta(\eta)\ge\eps+\beta(\eta)$ for some $\eps>0$, all
  $\eta\in\mathbb{S}^{d-2}$, and all $x$ with
  $D(x)\in[\delta_1,\delta_2]$. By Theorem~1 from \cite{mey:rei91} we
  have
  \begin{displaymath}
    0\leq \int_{\mathbb{S}^{d-2}}\langle \eta,y\rangle^2
    r(\eta)^d \cot\tilde\beta(\eta) \mathrm{d}\eta
    \leq \int_{\mathbb{S}^{d-2}}\langle \eta,y\rangle^2 
    r(\eta)^d \cot \beta(\eta) \mathrm{d}\eta
    -\eps\int_{\mathbb{S}^{d-2}}\langle \eta,y\rangle^2
    r(\eta)^d \mathrm{d}\eta,
  \end{displaymath}
  where we have used the inequality $\cot(\beta)-\cot(\beta+\eps)\geq 
  \eps$. Thus, 
  \begin{displaymath}
    R(x)\geq \Big(c_9\frac{\eps}{V_{d-1}(K\cap E)}\Big)^{d-1}. 
  \end{displaymath}
  The right-hand side is bounded away from zero for all $x$ with
  $D(x)\in[\delta_1,\delta_2]$. Thus, the curvatures of sets $K_t$
  with $t\in[\delta_1,\delta_2]$ are bounded by a constant, so that,
  $\|u_x-u_y\|\leq c_{10}\|x-y\|$ for a constant $c_{10}$ which
  depends on $K$ and $\delta_1,\delta_2$.

  Assume now that $t<s$.  Let $x\in\partial K_t$ with $u_x$ being the
  normal to $K_t$ at $x$. Let $z$ be the point on $K_s$ with the
  normal $u_x$. Since
  $\langle y,u_x\rangle\le\langle z,u_x\rangle=h(K_s,u_x)$, we have
  that
  \begin{math}
    h(K_t,u_x)-h(K_s,u_x)\leq \|x-y\|. 
  \end{math}
  Since $x$ and $z$ are the centroids of the sections
  $K\cap (x+\partial H_{u_x})$ and $K\cap (z+\partial H_{u_x})$,
  respectively, and $K_t$ is a subset of $K^{-\delta}$ for some
    $\delta>0$, 
  Lemma~\ref{lemma:1} yields that $\|x-z\|\leq c_2\|x-y\|$. Hence, $\|z-y\|\leq
  c_{11}\|x-y\|$, and, by the boundedness of curvatures of convex floating
  bodies, we have that $\|u_x-u_y\|=\|u_z-u_y\|\leq c_{12}\|x-y\|$. 
\end{proof}

\begin{theorem}
  \label{thr:uniform}
  Assume that $\mu$ is the uniform distribution on a centrally
  symmetric smooth strictly convex body $K$. Then the depth function
  is continuously differentiable on the interior of $K$ and 
  \begin{displaymath}
    \grad D(x)=\frac{V_{d-1}(K\cap(x+\partial H_{u_x}))}{V_d(K)} u_x. 
  \end{displaymath}
\end{theorem}
\begin{proof}
  The proof follows the argument from the proof of
  Lemma~\ref{lem:varpi}. Namely, 
  \begin{displaymath}
    D(x)-D(y)=\mu(x+H_{u_x})-\mu(y+H_{u_y})
    =\Big(\mu(x+H_{u_x})-\mu(y+H_{u_x})\Big)+ \Big(\mu(y+H_{u_x})-\mu(y+H_{u_y})\Big).
  \end{displaymath}
  The first term is continuously differentiable, and the second one is
  of the order $\|x-y\|^2$ uniformly in $x$ due to
  Lemmas~\ref{lemma:2} and \ref{lemma:3}.
\end{proof}

% Lemma~\ref{lem:varpi} also yields the similar expression of the
% gradient in case of a non-uniform absolutely continuous distribution,
% provided conditions (C3) and (C4) are satisfied.

\section*{Acknowledgement}
\label{sec:acknowledgement}

The authors are grateful to Lutz D\"umbgen for enlightening
discussions about empirical measures. IM and RT have been supported by
the Swiss National Science Foundation Grant No. 200021\_175584. 

%\bibliographystyle{abbrv}
%\bibliography{RT}

\end{document}